\documentclass[twoside, 11pt]{article}

\usepackage[utf8]{inputenc}
\usepackage[T1]{fontenc}
\usepackage[english]{babel}
\usepackage{amsmath, amssymb, amsthm, amstext, amsfonts, a4}
\usepackage{dsfont}
\usepackage[svgnames]{xcolor}
\usepackage{graphics,graphicx,subfigure}
\usepackage{comment}
\usepackage[affil-it]{authblk}
\usepackage{empheq}
\usepackage{diagbox}
\usepackage{esint}

\usepackage{hyperref}
\hypersetup{%
colorlinks=true,
linkcolor=Navy,
citecolor=Brown,
urlcolor=Navy
}

\usepackage{geometry}
\numberwithin{equation}{section}

\theoremstyle{plain}
	\newtheorem{theorem}{Theorem}[section] 
	\newtheorem{proposition}[theorem]{Proposition}       
	\newtheorem{lemma}[theorem]{Lemma}
	\newtheorem{corollary}[theorem]{Corollary}
        
\theoremstyle{definition}
	\newtheorem{definition}[theorem]{Definition}
    
	\newtheorem{remark}{Remark}[section]

\theoremstyle{remark}

\renewenvironment{proof}{\smallskip\noindent\emph{\textbf{Proof.}}%
  \hspace{1pt}}{\hspace{-5pt}{\nobreak\quad\nobreak\hfill\nobreak%
    $\square$\vspace{2pt}\par}\smallskip\goodbreak}
\newcommand{\limit}[2]{{\ \underset{#1 \to #2}{\longrightarrow} \ }}
\newcommand{\bigO}[1]{\ensuremath{\mathop{}\mathopen{}O\mathopen{}\left(#1\right)}}

\newcommand{\ds}[1]{\displaystyle{#1}}

\newcommand{\1}{\mathbf{1}} \renewcommand{\d}[1]{\mathinner{\mathrm{d}{#1}}} 
\newcommand{\p}{\partial} \newcommand{\eps}{\mathrm{\varepsilon}}

\newcommand{\sgn}{\mathop{\rm sgn}}

\newcommand{\N}{\mathbb{N}} \newcommand{\Z}{\mathbb{Z}} \newcommand{\R}{\mathbb{R}} 

\newcommand{\Ck}[1]{\mathbf{C}^{#1}} 
\newcommand{\Cc}[1]{\mathbf{C}_\mathbf{c}^{#1}}

\renewcommand{\L}[1]{\mathbf{L}^{#1}} 
\newcommand{\Lloc}[1]{\mathbf{L}_{\mathbf{loc}}^{#1}} 
\newcommand{\W}[2]{\mathbf{W}^{#1, #2}} 
\newcommand{\Wloc}[2]{\mathbf{W}_{\mathbf{loc}}^{#1, #2}} 
\newcommand{\BV}{\mathbf{BV}} 

\newcommand{\TV}{\mathbf{TV}}

\newcommand{\bF}{\mathbf{F}} 
\newcommand{\god}{\mathbf{God}}
\newcommand{\eo}{\mathbf{EO}} 

\newcommand{\cC}{\mathcal{C}}
\newcommand{\cG}{\mathcal{G}}
\newcommand{\cR}{\mathcal{R}}
\newcommand{\cP}{\mathcal{P}}

\DeclareFontFamily{U}{mathx}{\hyphenchar\font45}
\DeclareFontShape{U}{mathx}{m}{n}{
      <5> <6> <7> <8> <9> <10>
      <10.95> <12> <14.4> <17.28> <20.74> <24.88>
      mathx10
      }{}
\DeclareSymbolFont{mathx}{U}{mathx}{m}{n}
\DeclareFontSubstitution{U}{mathx}{m}{n}
\DeclareMathAccent{\widecheck}{0}{mathx}{"71}

\begin{document}

\title{\textbf{A LWR model with constraints at moving interfaces}}

\author{Abraham Sylla$^1$}

\date{}

\maketitle

\footnotetext[1]{\texttt{Abraham.Sylla@lmpt.univ-tours.fr} \\
Institut Denis Poisson, CNRS UMR 7013, Université de Tours, Université d'Orléans \\
Parc de Grandmont, 37200 Tours cedex, France \\
ORCID number: 0000-0003-1784-4878}

\thispagestyle{empty}

\begin{abstract}
    We propose a mathematical framework to the study of scalar conservation laws with moving interfaces. This framework is developed on a LWR model with 
    constraint on the flux along these moving interfaces. Existence is proved by means of a finite volume scheme. The originality lies in the local modification 
    of the mesh and in the treatment of the crossing points of the trajectories.

\end{abstract}



\tableofcontents

\newpage

\clearpage
\pagenumbering{arabic}

\section{Introduction}

Being given a regular concave flux $f \in \Ck{2}([0,1], \R^+)$ verifying
\begin{equation}
	\label{bell_shaped}
    f(0) = f(1) = 0, 
    \quad \exists! \; \overline{\rho} \in \mathopen]0, 1\mathclose[, \; 
    \text{for a.e.} \; \rho \in \mathopen]0, 1\mathclose[, \quad 
    f'(\rho)(\overline{\rho}-\rho) > 0,
\end{equation}

and a finite family of trajectories $(y_i)_{i \in [\![1;J]\!]}$ and constraints $(q_i)_{i \in [\![1;J]\!]}$ defined on $\mathopen]s_i, T_i \mathclose[$ ($0 \leq s_i < T_i$), 
we tackle the following problem:
\begin{equation}
	\label{multibus}
	\left\{
		\begin{array}{rlr}
			\p_{t}\rho(x, t) + \p_{x} \left( f(\rho(x, t)) \right) & = 0 & 
            (x, t) \in \R \times \mathopen]0,+\infty\mathclose[ := \Omega \\[5pt]
			\rho(x, 0) & = \rho_o(x) & x \in \R \\[5pt]
			\forall i \in [\![1;J]\!], \quad \left. \left(f(\rho) - \dot y_i(t) \rho \right) \right|_{x= y_i(t)} & \leq q_i(t) & t \in \mathopen]s_i, T_i \mathclose[.
		\end{array}
	\right.
\end{equation}

Systems of the type \eqref{multibus} have naturally arisen in the recent years. Let us give a non-exhaustive review on how our Problem \eqref{multibus} 
relates to the existing literature. 
\begin{itemize}
    \item The authors of \cite{DMLPSW2019, GGLP2020} considered a model very similar to \eqref{multibus}. In their framework, $(y_i)_i$ represented the 
    trajectories of autonomous vehicles, and the authors aimed at modeling the regulation impact on a few autonomous vehicles on the traffic flow. In the same 
    framework but with different applications in mind, the model of \cite{LBCG2021} accounts for the boundedness of traffic acceleration. Note that in each of 
    these models, the trajectories of the moving interfaces $(y_i)_i$ were not given \textit{a priori}, but rather obtained as solutions to an ODE involving the 
    density of traffic, a mechanism reminiscent of \cite{AGS2010, DMG2014, Sylla2020} for instance. Let us also mention the work of \cite{GLM2013} where 
    the authors studied a different model for the situation of several moving bottlenecks.
    \item The numerical aspect of \eqref{multibus} was treated in \cite{CDMG2018} (for one trajectory) and \cite{DMG2016} (for multiple trajectories), where the 
    authors modeled the moving bottlenecks created by buses on a road. 
    \item In a class of problems close to \eqref{multibus}, \textit{i.e.} without constraint on the flux, but still with coupling interfaces/density, the authors of 
    \cite{FGP2020} described the interaction between a platoon of vehicles and the surrounding traffic flow on a highway. 
    \item Problem \eqref{multibus} can be seen as a conservation law with discontinuous flux and special treatments at the interfaces. In that directions, the 
    authors of \cite{KT2004, AM2015, Andreianov2015, BGS2019, Towers2020} studied such problems but with the classical vanishing viscosity coupling at 
    the interfaces.
\end{itemize}

In several of these works \cite{GGLP2020, LBCG2021}, the existence issue is tackled using the wave-front tracking procedure which is very sensible to the 
details of the model. On the other hand, when numerical schemes are considered, see \cite{DMG2016,CDMG2018}, the numerical analysis is usually left out.
\smallskip 

The contribution of this paper is to provide a robust mathematical setting both in the theoretical and numerical aspects of \eqref{multibus}. The proof of 
uniqueness is based upon a combination of Kruzhkov classical method of doubling variables and the theory of dissipative germs in the framework of discontinuous 
flux \cite{AKR2011}, and it is analogous to the one of \cite{AM2015}. To prove existence, we build a finite volume scheme with a grid that adapts locally to 
the trajectories $(y_i)_i$ and to their crossing points, but remains a simple Cartesian grid away from the interfaces. Our work can serve as a basis for 
constructing solutions to more involved models, \textit{e.g.} \textit{via} the splitting approach. As an example of application, we can point out the variant of our recent 
work \cite{Sylla2020} with multiple slow vehicles involved; this is a mildly non-local analogue of the problem considered numerically in \cite{DMG2016}.
\bigskip

As the fundamental ingredient of the well-posedness proof and numerical approximation of \eqref{multibus}, we will first tackle the one trajectory/one 
constraint problem:
\begin{equation}
	\label{1Bus}
	\left\{
		\begin{aligned}
			\p_{t}\rho + \p_{x} \left( f(\rho) \right) & = 0 \\[5pt]
			\rho(\cdot, 0) & = \rho_o \\[5pt]
			\left. \left(f(\rho) - \dot y(t) \rho \right) \right|_{x= y(t)} & \leq q(t) & t > 0,
		\end{aligned}
	\right.
\end{equation}

with $y \in \Wloc{1}{\infty}(\mathopen]0, +\infty\mathclose[, \R)$ and $q \in \Lloc{\infty}(\mathopen]0, +\infty\mathclose[, \R)$. Models in the class of \eqref{1Bus} have been greatly investigated in the past 
few decades. Motivated by the modeling of tollgates and traffic lights for instance, the authors of \cite{CG2007} considered \eqref{1Bus} with the trivial 
trajectory $y \equiv 0$ and proved a well-posedness result in the $\BV$ framework (\textit{i.e.} with both $q$ and $\rho_o$ with bounded variation, locally). The authors 
of \cite{AGS2010} then extended the well-posedness in the $\L{\infty}$ framework and also constructed a convergent numerical scheme. More recently, in 
\cite{DMG2014, DMG2017, Sylla2020}, the authors studied a variant of \eqref{1Bus} in which $\rho$ and $\dot y$ were coupled \textit{via} an ODE. The coupling was thought 
to model the influence of a slow vehicle, traveling at speed $\dot y$, on road traffic.
\smallskip 

The reduction of \eqref{multibus} to localized problem \eqref{1Bus} requires the construction of a finite volume scheme in the original coordinates $(x, t)$, 
while the treatment of \eqref{1Bus} in the literature is most often based upon the rectification of the interface \textit{via} a variable change, see 
\cite{DMG2014, DMG2017, Sylla2020}. For \eqref{multibus}, this approach leads to a cumbersome and singular construction, see \cite{AM2015}. In our well-posedness 
analysis and approximation of \eqref{1Bus}, having in mind \eqref{multibus}, we will not change the coordinate system.
\smallskip 

Let us detail how the paper is organized. Sections \ref{1Bus_Section_uniqueness}-\ref{1Bus_Section_existence} are devoted to Problem \eqref{1Bus}. We start by 
giving two definitions of solutions. One, most frequently used in traffic dynamics (see \cite{CG2007, AS2020}), is composed of classical Kruzhkov entropy 
inequalities with reminder term taking into account the constraint and of a weak formulation for the constraint, see Definition \ref{1Bus_def1}. The second 
definition emanates from the theory of conservation laws with dissipative interface coupling (see \cite{AKR2011,Andreianov2015}). It consists of Kruzhkov 
entropy inequalities with test functions that vanish along the interface $\{x = y(t)\}$ and of an explicit treatment of the traces of the solution along the 
interface, see Definition \ref{1Bus_def2}. Before tackling the well-posedness issue, we prove that these two definitions are equivalent, see Propositions 
\ref{1Bus_def21_pp}-\ref{1Bus_def21_pp}, similarly to what the authors of \cite{AGS2010} did. Uniqueness follows from the stability obtained in 
Section \ref{1Bus_Section_uniqueness}, see Theorem \ref{1Bus_stability1}. In Section \ref{1Bus_Section_existence}, we construct a finite volume 
scheme for \eqref{1Bus} and prove of its convergence. In the construction, we do not rectify the trajectory, but instead we locally modify the mesh to mold 
the trajectory. Moreover, we fully make use of techniques and results put forward by the author of \cite{TowersOSLC} to derive localized $\BV$ estimates away 
from the interface, essential to obtain strong compactness for the approximate solutions created by the scheme, see Corollary \ref{1Bus_oslc3_cor}. This is a 
way to highlight the generality of the compactness technique of \cite{TowersOSLC}.

In Section \ref{multibus_Section_wp}, we get back to the original problem \eqref{multibus}. Our strategy is to \textit{assemble} the study of 
\eqref{multibus} from several local studies of \eqref{1Bus} with the help of a partition of unity argument. This concerns, in particular, the convergence of 
finite volume approximation of \eqref{multibus} which is addressed \textit{via} a localization argument. However, the scheme needs to be defined globally, which makes it 
impossible to use the rectification strategy as soon as the interfaces have crossing points, see \cite{AM2015} for a singular rectification strategy.

\section{Uniqueness and stability for the single trajectory problem}
\label{1Bus_Section_uniqueness}

The content of this section is not original in the sense that it is a rigorous adaptation and assembling of existing techniques reminiscent of 
\cite{Volpert1967, Kruzhkov1970, CG2007, AGS2010, AKR2011}.

\subsection{Equivalent definitions of solutions}

Throughout the paper, for all $s \in \R$, we denote by
\[
    \forall \rho \in [0,1], \; F_s(\rho) := f(\rho) - s\rho \quad \text{and} \quad \forall a,b \in [0,1], \; \Phi_s(a,b) := \sgn(a-b)(F_s(a) - F_s(b)),
\]

the normal flux through $\{x = x_o + st\}$ ($x_o \in \R$) and its entropy flux associated with the Kruzhkov entropy $\rho \mapsto |\rho-\kappa|$, 
for all $\kappa \in [0,1]$, see \cite{Kruzhkov1970}. Let us also denote by $\Gamma$ the trajectory/interface:
\[
    \Gamma := \{(y(t), t) \; : \; t \in [0, +\infty[\}.  
\]

\begin{definition}
    \label{1Bus_def1}
    Let $\rho_o \in \L{\infty}(\R, [0, 1])$. We say that $\rho \in \L{\infty}(\Omega, [0, 1])$ is an admissible entropy solution to \eqref{1Bus} if

    (i) for all test functions $\varphi \in \Cc{\infty}(\overline{\Omega}, \R^+)$ and $\kappa \in [0,1]$, the following entropy inequalities are 
    verified:
	\begin{equation}
        \label{1Bus_ei1}
        \begin{aligned}
            \int_{0}^{+\infty} \int_{\R} & \biggl( |\rho-\kappa| \p_{t}\varphi + \Phi(\rho,\kappa) \p_{x} \varphi \biggr) \d x \d t  
            + \int_{\R} |\rho_o(x)-\kappa|\varphi(x,0) \d x \\[5pt]
            & + \int_{0}^{+\infty} \cR_{\dot y(t)}(\kappa,q(t)) \varphi(y(t),t) \d t \geq 0,
        \end{aligned}
	\end{equation}

    where
    \[
        \cR_{\dot y(t)}(\kappa, q(t)) :=  2 \left( F_{\dot y(t)}(\kappa) - \min \left\{ F_{\dot y(t)}(\kappa), q(t) \right\} \right);
    \]

    (ii) for all test functions $\varphi \in \Cc{\infty}(\Omega, \R^+)$ the following constraint inequalities are verified:
	\begin{equation}
        \label{1Bus_ci1}
        - \iint_{\Omega^+} \biggl( \rho \p_{t} \varphi + f(\rho) \p_x \varphi \biggr) \d x \d t \leq \int_{0}^{+\infty} q(t) \varphi(y(t), t) \d t,
    \end{equation}

    where $\ds{\Omega^+ := \left\{ (x,t) \in \Omega \; : \; x > y(t) \right\}}$.
\end{definition}

\begin{remark}
    \label{1Bus_rk1}
    Taking $\kappa = 0$, then $\kappa = 1$ in \eqref{1Bus_ei1}, from the condition $\rho(x,t) \in [0,1]$ a.e. we deduce that any admissible weak solution to 
    Problem \eqref{1Bus} is also a distributional solution to the conservation law $\p_{t}\rho + \p_{x} f(\rho) = 0$. If $\rho$ is a regular enough solution, 
    then for all test functions $\varphi \in \Cc{\infty}(\Omega, \R^+)$, we have 
    \[
        \begin{aligned}
            0 
            & = \iint_{\Omega^+} \text{div}_{(x,t)} \begin{pmatrix} f(\rho) \\ \rho \end{pmatrix} \varphi \ \d{x} \d{t} \\[5pt]
            & = \int_{\p \Omega^+} \begin{pmatrix} f(\rho) \varphi \\ \rho \varphi \end{pmatrix} \cdot 
            \begin{pmatrix} -1 \\ \dot y(t) \end{pmatrix} \d{t} 
            - \iint_{\Omega^+} \begin{pmatrix} f(\rho) \\ \rho \end{pmatrix} \cdot \nabla_{x,t} \varphi \ \d{x} \d{t} \\[5pt]
            & = - \int_0^{+\infty} \biggl( \left( f(\rho) - \dot y(t) \rho \right)_{|x = y(t)} \biggr) \varphi(y(t), t) \d{t}
            - \iint_{\Omega^+} \biggl( \rho \p_{t} \varphi + f(\rho) \p_x \varphi \biggr) \d x \d t.
        \end{aligned}
    \]

    Moreover, if $\rho$ satisfies the flux inequality of \eqref{1Bus} a.e. on $\mathopen]0, +\infty\mathclose[$, then the previous computations lead to 
    \[
        - \iint_{\Omega^+} \biggl( \rho \p_{t} \varphi + f(\rho) \p_x \varphi \biggr) \d x \d t \leq \int_0^{+\infty} q(t) \varphi(y(t), t) \d{t};
    \]

    this is where inequalities \eqref{1Bus_ci1} come from. Note how they make sense irrespective of the regularity of $\rho$. Integrating on 
    $\ds{\Omega^- := \left\{ (x,t) \in \Omega \; : \; x < y(t) \right\}}$ would lead to similar and equivalent inequalities.
\end{remark}

Definition \ref{1Bus_def1} is well suited for passage to the limit of a.e. convergent sequences of exact or approximate solutions. However, we cannot derive 
uniqueness by the standard arguments like in the classical case of Kruzhkov. Using an equivalent notion of solution, which we adapt from \cite{AKR2011}, based on 
explicit treatment of traces of $\rho$ on $\Gamma$, we rather combine the arguments of \cite{Kruzhkov1970} and \cite{Volpert1967}. In this definition a couple 
plays a major role, the one which realizes the equality in the flux constraint in \eqref{1Bus}. More precisely, fix first $s \geq 0$. By \eqref{bell_shaped} 
and concavity of $f$, for all $q \in [0, \max F_s)$, the equation $F_s(\rho) = q$ admits exactly two solutions in $[0, 1]$, see Figure \ref{fig1}, left. 
The same way, if $s \leq 0$, then for all $q \in [-\dot s, \max F_{s})$, the equation still admits two solutions in $[0, 1]$. The couple formed by these two 
solutions, 
denoted by $\left( \widehat{\rho}_{s}(q), \widecheck{\rho}_{s}(q) \right)$ in Definition \ref{1Bus_def_germ} below, will serve both in the proof of uniqueness and 
existence.

\begin{figure}[!htp]
    \begin{center}
        \includegraphics[scale = 0.66]{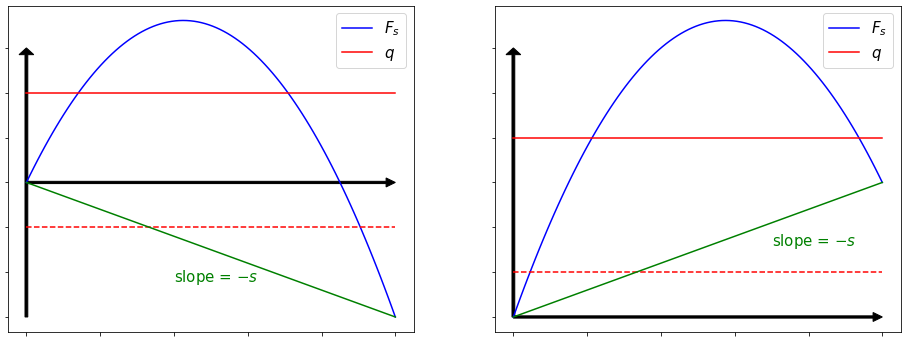}
    \end{center}
    \vspace*{-5mm}
    \caption{Illustration of Assumption \eqref{1Bus_assumption_constraint}}
    \label{fig1}
\end{figure}

Following the previous discussion, in the sequel, we will assume that $q$ verifies the following assumption:
\begin{equation}
    \label{1Bus_assumption_constraint}
    \text{for a.e.} \; t > 0, \quad 
    q(t) \in [0, \max F_{\dot y(t)}\mathclose[ \; \text{if} \; \dot y(t) \geq 0 
    \;\; \text{and} \;\; 
    q(t) \in [-\dot y(t), \max F_{\dot y(t)}\mathclose[ \; \text{if} \; \dot y(t) < 0.
\end{equation}

In particular, remark that 
\begin{equation}
    \label{1Bus_assumption_constraint2}
    \text{for a.e.} \; t > 0, \quad \dot y(t) + q(t) \geq 0.
\end{equation}

\begin{definition}
    \label{1Bus_def_germ}
    Let $s \in \R^+$ and $q \in [0, \max F_s\mathclose[$, or $s \in \R^-$ and 
    $q \in [-s, \max F_s\mathclose[$. The admissibility germ for the conservation law in \eqref{1Bus} 
    associated with the constraint $\ds{F_s(\rho)_{|x = st} \leq q}$ is the subset $\cG_s(q) \subset [0,1]^2$ defined as the union: 
    \[
        \cG_s(q) := 
        \underbrace{\left( \widehat{\rho}_{s}(q), \widecheck{\rho}_{s}(q) \right)}_{\cG^1_s(q)} 
        \bigcup \underbrace{\{(\kappa, \kappa) \; : \; F_s(\kappa) \leq q\}}_{\cG^2_s(q)} 
        \bigcup \underbrace{\left\{ (k_l, k_r) \; : \; k_l < k_r \; \text{and} \; 
        F_s(k_l) = F_s(k_r) \leq q \right\}}_{\cG^3_s(q)},
    \]

    where, due to the bell-shaped profile of $F_s$, the couple $\left( \widehat{\rho}_{s}(q), \widecheck{\rho}_{s}(q) \right)$ is uniquely defined by the 
    conditions
	\[
		F_s(\widehat{\rho}_{s}(q)) = F_s(\widecheck{\rho}_{s}(q)) = q \quad \text{and} \quad \widehat{\rho}_{s}(q) > \widecheck{\rho}_{s}(q).
	\]
\end{definition}

\begin{lemma}
    \label{1Bus_germ_dissipativity}
    For all $s \in \R^+$ and $q \in [0, \max F_s\mathclose[$, and for all $s \in \R^-$ and $q \in [-s, \max F_s\mathclose[$, the admissibility germ $\cG_s(q)$ is $\L{1}$-dissipative 
    in the sense that:

    (i) for all $(k_l, k_r) \in \cG_s(q)$, $F_s(k_l) = F_s(k_r)$ (Rankine-Hugoniot condition);

    (ii) for all $(k_l, k_r), (c_l, c_r) \in \cG_s(q)$,
    \begin{equation}
        \label{1Bus_L1D}
        \Phi_s(k_l, c_l) \geq \Phi_s(k_r, c_r).
    \end{equation}
\end{lemma}

\begin{proof}
    The point \textit{(i)} is obvious from the definition. Let us prove the dissipative 
    feature \eqref{1Bus_L1D}. The following table summarizes which values can take the 
    difference $\Delta = \Phi_s(k_l, c_l) - \Phi_s(k_r, c_r)$, depending on in 
    which parts of the germ the couples $(k_l, k_r), (c_l, c_r) \in \cG_s(q)$ belong 
    to. 
    \begin{center}
        \renewcommand{\arraystretch}{1.5}
        \begin{tabular}{|c||c|c|c|}
            \hline
            \diagbox[width = 10em]{$(c_l, c_r)$}{$(k_l, k_r)$}
            & $\in \cG_s^1(q)$ & $\in \cG_s^2(q)$ & $\in \cG_s^3(q)$ \\[5pt]
            \hline
            $\in \cG_s^1(q)$ & $0$ & $0$ & $0$ or $2(q - F_s(k_l))$ \\[5pt]
            \hline
            $\in \cG_s^2(q)$ & $0$ & $0$ & $0$ or $2|F_s(c) - F_s(k_l)|$ \\[5pt]
            \hline
            $\in \cG_s^3(q)$ & $0$ or $2(q - F_s(c_l))$ & $0$ or $2|F_s(c_l) - F_s(k)|$ & $0$ or $2|F_s(c_l) - F_s(k_l)|$ \\[5pt]
            \hline
        \end{tabular}
    \end{center}

    Having in mind the definition of $\cG^3_s(q)$, we can conclude that $\Delta \geq 0$.
\end{proof}

\begin{definition}
    \label{1Bus_def2}
    Let $\rho_o \in \L{\infty}(\R, [0, 1])$. We say that $\rho \in \L{\infty}(\Omega, [0, 1])$ is a $\cG_{\dot y}(q)$-entropy solution to \eqref{1Bus} if:

    (i) for all test functions $\varphi \in \Cc{\infty}(\overline{\Omega} \backslash \Gamma, \R^+)$ and $\kappa \in [0,1]$, the following entropy
    inequalities are verified:
	\begin{equation}
		\label{1Bus_ei2}
        \int_{0}^{+\infty} \int_{\R} \biggl( |\rho-\kappa| \p_{t}\varphi + \Phi(\rho,\kappa) \p_{x} \varphi \biggr) \d x \d t 
        + \int_{\R}|\rho_o(x)-\kappa|\varphi(x,0) \d x \geq 0;
    \end{equation}

    (ii) for a.e. $t > 0$,
	\begin{equation}
        \label{1Bus_ci2}
        (\rho(y(t)-,t), \rho(y(t)+,t)) \in \cG_{\dot y(t)}(q(t)).
    \end{equation}
\end{definition}

\begin{remark}
    Condition \eqref{1Bus_ci2} is to be understood in the sense of strong traces along $\Gamma$. An important fact we stress is that it is not 
    restrictive to assume that entropy solutions, \textit{i.e.} bounded functions verifying \eqref{1Bus_ei2}, admit strong traces. Usually, it is ensured provided a 
    nondegeneracy assumption on the flux function:
    \begin{equation}
        \label{1Bus_nonD}
        \text{for any nonempty interval} \; \mathopen]a, b \mathclose[ \subset 
        \mathopen]0, 1 \mathclose[, \quad f \; \1_{\mathopen]a, b \mathclose[} \; \text{is not constant.}
    \end{equation}

    In the context of traffic flow, however, we sometimes consider fluxes which do not verify \eqref{1Bus_nonD}. Such fluxes, which have linear parts, usually 
    model constant traffic velocity for small densities. In those situations, and when $y \equiv 0$, one can prove that under a mild assumption on the 
    constraint, if the initial datum has bounded variation, then solutions to \eqref{1Bus} are in $\L{\infty}(\mathopen]0, T \mathclose[, \BV(\R))$, and traces are then to be 
    understood in the sense of $\BV(\R)$ functions, see \cite[Theorem 3.2]{Sylla2020}. Also note that the germ formalism can be adapted to the situations where 
    the flux is degenerate and no variation bound is assumed, see \cite[Remarks 2.2, 2.3]{AKR2011}.
\end{remark}

We now prove that Definitions \ref{1Bus_def1} and \ref{1Bus_def2} are equivalent.

\begin{proposition}
    \label{1Bus_def12_pp}
    Any admissible entropy solution to \eqref{1Bus} is a $\cG_{\dot y}(q)$-entropy solution.
\end{proposition}

\begin{proof}
    Fix $\rho_o \in \L{\infty}(\R, [0, 1])$ and let 
    $\rho \in \L{\infty}(\Omega, [0, 1])$ be an admissible entropy solution 
    to \eqref{1Bus}. Let $\varphi \in \Cc{\infty}(\overline{\Omega}, \R^+)$ and 
    $\kappa \in [0,1]$. If $\varphi$ vanishes along $\Gamma$, then \eqref{1Bus_ei1} 
    becomes \eqref{1Bus_ei2}. Moreover, it is known that the Rankine-Hugoniot condition is contained in \eqref{1Bus_ei1}. Combining it with \eqref{1Bus_ci1} gives us:
    \begin{equation}
        \label{1Bus_def12_eq1}
        \text{for a.e.} \; t > 0, \quad F_{\dot y(t)}(\rho(y(t)-,t)) = F_{\dot y(t)}(\rho(y(t)+,t)) \leq q(t).
    \end{equation}

    Let us show that for a.e. $t > 0$, $(\rho(y(t)-,t), \rho(y(t)+,t)) \in \cG_{\dot y(t)}(q(t))$.

    \textbf{Case 1:} $\rho(y(t)-,t) \leq \rho(y(t)+,t)$. Condition \eqref{1Bus_def12_eq1} implies that 
    $(\rho(y(t)-,t), \rho(y(t)+,t)) \in \cG^2_{\dot y(t)}(q(t)) \cup \cG^3_{\dot y(t)}(q(t))$. 
    
    \textbf{Case 2:} $\rho(y(t)-,t) > \rho(y(t)+,t)$. Suppose now that 
    $\varphi \in \Cc{\infty}(\Omega, \R^+)$ and fix $n \in \N^*$. By a standard 
    approximation argument, we can apply \eqref{1Bus_ei1} with the Lipschitz test function $\xi_n \varphi$, where $\xi_n$ is the cut-off function:
    \[
        \xi_n(x,t) = 
        \left\{
            \begin{array}{ccc}
                1 & \text{if} & \ds{|x - y(t)| < \frac{1}{n}} \\[5pt]
                2 - n|x-y(t)| & \text{if} & \ds{\frac{1}{n} \leq |x - y(t)| \leq \frac{2}{n}} \\[5pt]
                0 & \text{if} & \ds{|x - y(t)| > \frac{2}{n}}.
            \end{array}
        \right.
    \]

    This yields:
    \[
        \begin{aligned}
            & \int_{0}^{+\infty} \int_{\R} |\rho-\kappa| \left(\xi_n \p_{t}\varphi 
            + n \dot y(t) \sgn(x - y(t)) \1_{\left\{\frac{1}{n} < |x - y(t)| < \frac{2}{n} \right\}} \varphi \right) \d{x} \d{t} \\[5pt] 
            & + \int_{0}^{+\infty} \int_{\R} \Phi(\rho,\kappa) 
            \left( \xi_n \p_{x} \varphi - n \sgn(x-y(t) \1_{\left\{\frac{1}{n} < |x - y(t)| < \frac{2}{n} \right\}} \varphi \right) \d x \d t \\[5pt]
		    & + \int_{0}^{+\infty} \cR_{\dot y(t)}(\kappa,q(t)) \varphi(y(t),t) \d t \geq 0.
        \end{aligned}
    \]

    Taking the limit when $n \to +\infty$, we obtain:
    \[
        \int_0^{+\infty} \biggl( \Phi_{\dot y(t)} \left(\rho(y(t)-,t), \kappa \right) - \Phi_{\dot y(t)} \left(\rho(y(t)+,t), \kappa \right) 
        + \cR_{\dot y(t)}(\kappa,q(t)) \biggr) \varphi(y(t), t) \d{t} \geq 0,
    \]

    which implies that for a.e. $t > 0$ and for all $\kappa \in [0,1]$,
    \[
        \Phi_{\dot y(t)} \left(\rho(y(t)-,t), \kappa \right) - \Phi_{\dot y(t)} \left(\rho(y(t)+,t), \kappa \right) + \cR_{\dot y(t)}(\kappa,q(t)) \geq 0.
    \]

    Taking in particular $\kappa = \text{argmax}(F_{\dot y(t)})$, we get:
    \begin{equation}
        \label{1Bus_def12_eq2}
        \Phi_{\dot y(t)} \left(\rho(y(t)-,t), \kappa \right) - \Phi_{\dot y(t)} \left(\rho(y(t)+,t), \kappa \right) + 2 (F_{\dot y(t)}(\kappa) - q(t)) \geq 0.
    \end{equation}

    Since $\rho(y(t)-,t) > \rho(y(t)+,t)$, \eqref{1Bus_def12_eq2} leads to $\ds{F_{\dot y(t)}(\rho(y(t)-,t)) \geq q(t)}$, which combined with 
    \eqref{1Bus_def12_eq1}, implies $\ds{F_{\dot y(t)}(\rho(y(t)-,t)) = F_{\dot y(t)}(\rho(y(t)+,t)) = q(t)}$. We deduce that 
    $(\rho(y(t)-,t), \rho(y(t)+,t)) \in \cG^1_{\dot y(t)}(q(t))$, which completes the proof.
\end{proof}

\begin{proposition}
    \label{1Bus_def21_pp}
    Any $\cG_{\dot y}(q)$-entropy solution to \eqref{1Bus} is an admissible entropy solution.
\end{proposition}

\begin{proof}
    Fix $\rho_o \in \L{\infty}(\R, [0, 1])$ and let 
    $\rho \in \L{\infty}(\Omega, [0, 1])$ be a $\cG_{\dot y}(q)$-entropy solution to \eqref{1Bus}. Let $\varphi \in \Cc{\infty}(\overline{\Omega}, \R^+)$ and 
    $\kappa \in [0,1]$. We still denote by $\xi_n$ the cut-off function from the last proof. We write $\ds{\varphi = (1 - \xi_n) \varphi + \xi_n \varphi}$. Since $\phi_n = (1 - \xi_n) \varphi$ vanishes along $\Gamma$, we have
    \[
        \begin{aligned}
            \mathbf{I} 
            & = \int_{0}^{+\infty} \int_{\R} \biggl( |\rho-\kappa| \p_{t}\varphi + \Phi(\rho,\kappa) \p_{x} \varphi \biggr) \d x \d t  
            + \int_{\R}|\rho_o(x)-\kappa|\varphi(x,0) \d x \\[5pt] 
            & + \int_{0}^{+\infty} \cR_{\dot y(t)}(\kappa,q(t)) \varphi(y(t),t) \d t \\[5pt]
            & = \underbrace{\int_{0}^{+\infty} \int_{\R} \biggl( |\rho-\kappa| \p_{t}\phi_n + \Phi(\rho,\kappa) \p_{x} \phi_n \biggr) \d x \d t  
            + \int_{\R}|\rho_o(x)-\kappa|\phi_n(x,0) \d x}_{\geq 0} \\[5pt]
            & + \int_{0}^{+\infty} \int_{\R} \biggl( |\rho-\kappa| \p_{t} (\xi_n \varphi) + \Phi(\rho,\kappa) \p_{x} (\xi_n \varphi) \biggr) \d x \d t  
            + \int_{\R}|\rho_o(x)-\kappa|\xi_n(x,0) \varphi(x,0) \d x \\[5pt] 
            & + \int_{0}^{+\infty} \cR_{\dot y(t)}(\kappa,q(t)) \varphi(y(t),t) \d t \\[5pt]
            & \geq \int_{0}^{+\infty} \int_{\R} |\rho-\kappa| \left(\xi_n \p_{t}\varphi 
            + n \dot y(t) \sgn(x - y(t)) \1_{\left\{\frac{1}{n} < |x - y(t)| < \frac{2}{n} \right\}} \varphi \right) \ \d{x} \d{t} \\[5pt]
            & + \int_{0}^{+\infty} \int_{\R} \Phi(\rho,\kappa) 
            \left( \xi_n \p_{x} \varphi - n \sgn(x-y(t) \1_{\left\{\frac{1}{n} < |x - y(t)| < \frac{2}{n} \right\}} \varphi \right) \ \d x \d t \\[5pt]
            & + \int_{\R}|\rho_o(x)-\kappa| \xi_n(x,0) \varphi(x,0) \d x + \int_{0}^{+\infty} \cR_{\dot y(t)}(\kappa,q(t)) \varphi(y(t),t) \d t.
        \end{aligned}
    \]

    Taking the limit when $n \to +\infty$, we obtain:
    \[
        \mathbf{I} 
        \geq \int_0^{+\infty} \biggl( \underbrace{\Phi_{\dot y(t)} \left(\rho(y(t)-,t), \kappa \right) - \Phi_{\dot y(t)} \left(\rho(y(t)+,t), \kappa \right) 
        + \cR_{\dot y(t)}(\kappa,q(t))}_{\Delta(t,\kappa)} \biggr) \varphi(y(t), t) \d{t}.
    \]

    To conclude, we are going to prove that for a.e. $t > 0$ and for all $\kappa \in [0,1]$, $\Delta(t,\kappa) \geq 0$. Remember that by assumption, for a.e. 
    $t > 0$, $(\rho(y(t)-,t), \rho(y(t)+,t)) \in \cG_{\dot y(t)}(q(t))$. The following table, in which we dropped the $\dot y(t)/q(t)$-indexing, summarizes 
    which values can take the difference $\Delta(t, \kappa)$ according to the position of $\kappa$ with respect to the couple $(\rho(y(t)-,t), \rho(y(t)+,t))$, 
    which is simply denoted by $(\rho_l, \rho_r)$. Note that the case marked by $\mathbf{\times}$ does not happen.

    \begin{center}
        \renewcommand{\arraystretch}{1.5}
        \begin{tabular}{|c||c|c|c|}
            \hline
            \diagbox[width = 10em]{$\kappa$}{$(\rho_l, \rho_r)$}
            & $\in \cG^1$ & $\in \cG^2$ & $\in \cG^3$ \\[5pt]
            \hline
            $\kappa < \min\{\rho_l, \rho_r\}$ & $0$ & $\cR(\kappa,q(t))$ & $0$ \\[5pt]
            \hline
            $\kappa > \max\{\rho_l, \rho_r\}$ & $0$ & $\cR(\kappa,q(t))$ & $0$ \\[5pt]
            \hline
            $\kappa$ between $\rho_l$ and $\rho_r$ & $0$ & $\mathbf{\times}$ & $ 2 (F(\kappa) - F(\rho_l)) + \cR(\kappa,q(t))$ \\[5pt]
            \hline
        \end{tabular}
    \end{center}

    Clearly, $\Delta(t, \kappa) \geq 0$, which proves that $\mathbf{I} \geq 0$, hence $\rho$ satisfies \eqref{1Bus_ei1}. Moreover, by assumption, for a.e. $t > 0$, 
    $(\rho(y(t)-,t), \rho(y(t)+,t)) \in \cG_{\dot y(t)}(q(t))$. This implies, in particular, that $\rho$ satisfies the flux constraint inequality
    $\ds{(f(\rho) - \dot y(t) \rho)_{|x = y(t)} \leq q(t)}$ in the a.e. sense. By Remark \ref{1Bus_rk1}, $\rho$ satisfies \eqref{1Bus_ci1} as well 
    \textit{i.e.} $\rho$ is an admissible entropy solution to \eqref{1Bus}. 
\end{proof}

\subsection{Uniqueness of $\cG$-entropy solutions}

We now prove uniqueness using Definition \ref{1Bus_def2}.

\begin{lemma}[Kato inequality]
    \label{1Bus_Kato1_lmm}
    Fix $\rho_o, \sigma_o \in \L{\infty}(\R, [0,1])$, 
    $y \in \Wloc{1}{\infty}(\mathopen]0, +\infty\mathclose[, \R)$, $q \in \Lloc{\infty}(\mathopen]0, +\infty\mathclose[, \R)$. We denote by $\rho$ a $\cG_{\dot y}(q)$-entropy solution to \eqref{1Bus}. The same way, let $\sigma$ be $\cG_{\dot y}(r)$-entropy solution to Problem \eqref{1Bus} with initial datum $\sigma_o$.
    We suppose that $q, r$ satisfy \eqref{1Bus_assumption_constraint}. Then for all test functions 
    $\varphi \in \Cc{\infty}(\overline{\Omega}, \R^+)$, we have
    \begin{equation}
        \label{1Bus_Kato1}
        \begin{aligned}
            & \int_{0}^{+\infty} \int_{\R} \biggl( |\rho- \sigma| \p_{t}\varphi + \Phi(\rho, \sigma) \p_{x} \varphi \biggr) \d x \d t  
            + \int_{\R} |\rho_o(x) - \sigma_o(x)| \varphi(x,0) \d x \\[5pt]
            & + \int_0^{+\infty} 
            \biggl( \Phi_{\dot y(t)} \left(\rho(y(t)+,t), \sigma(y(t)+,t) \right) - \Phi_{\dot y(t)} \left(\rho(y(t)-,t), \sigma(y(t)-,t) \right) \biggr)
            \varphi(y(t),t) \d{t} \geq 0.
        \end{aligned}
	\end{equation} 
\end{lemma}

\begin{proof}
    Take $\phi = \phi(x, t, \chi, \tau) \in \Cc{\infty}(\overline{\Omega}^2, \R^+)$ 
    with support contained in the set 
    $\ds{\left(\overline{\Omega} \backslash \Gamma \right)^2}$. The classical method of doubling variables leads us to:
    \begin{equation}
        \label{1Bus_Kato_eq1}
        \begin{aligned}
            & \iiiint |\rho(x,t) - \sigma(\chi, \tau)| (\p_t \phi + \p_\tau \phi) + \Phi(\rho(x,t), \sigma(\chi, \tau)) (\p_x \phi + \p_\chi \phi) \ 
            \d{x} \d{t} \d{\chi} \d{\tau} \\[5pt]
            & + \iiint |\rho_o(x) - \sigma(\chi, \tau)| \phi(x, 0, \chi, \tau) \d{x} \d{\chi} \d{\tau}
            + \iiint |\rho(x, t) - \sigma_o(\chi)| \phi(x, t, \chi, 0) \d{x} \d{t} \d{\chi} \geq 0.
        \end{aligned}
    \end{equation}

    Again, a standard approximation argument allows us to apply \eqref{1Bus_Kato_eq1} with the Lipschitz function
    \[
        \phi_n(x, t, \chi, \tau) 
        = \gamma_n(x,t) \varphi\left( \frac{x + \chi}{2}, \frac{t + \tau}{2} \right) \delta_n \left( \frac{x - \chi}{2} \right) 
        \delta_n \left( \frac{t - \tau}{2} \right),
    \]

    where $\varphi = \varphi(X, T) \in \Cc{\infty}(\overline{\Omega}, \R^+)$, 
    $(\delta_n)_n$ is a smooth approximation of the Dirac mass at the origin, and 
    \[
        \gamma_n(x,t) = 
        \left\{
            \begin{array}{ccc}
                0 & \text{if} & \ds{|x - y(t)| < \frac{1}{n}} \\[5pt]
                \ds{n \left(|x-y(t)| - \frac{1}{n} \right)} & \text{if} & \ds{\frac{1}{n} \leq |x - y(t)| \leq \frac{2}{n}} \\[5pt]
                1 & \text{if} & \ds{|x - y(t)| > \frac{2}{n}}.
            \end{array}
        \right.
    \]

    Using the fact that for a.e. $t > 0$,
    \[
        \begin{aligned}
            \p_t \phi_n + \p_\tau \phi_n
            & = -n \dot y(t) \sgn(x - y(t)) \1_{\left\{\frac{1}{n} < |x - y(t)| < \frac{2}{n} \right\}} 
            \varphi \left( \frac{x + \chi}{2}, \frac{t + \tau}{2} \right) 
            \delta_n \left( \frac{x - \chi}{2} \right) \delta_n \left( \frac{t - \tau}{2} \right) \\[5pt]
            & + \gamma_n(x,t) \p_T \varphi\left( \frac{x + \chi}{2}, \frac{t + \tau}{2} \right) 
            \delta_n \left( \frac{x - \chi}{2} \right) \delta_n \left( \frac{t - \tau}{2} \right) \\[5pt]
            \p_x \phi_n + \p_\chi \phi_n
            & = n \sgn(x - y(t)) \1_{\left\{\frac{1}{n} < |x - y(t)| < \frac{2}{n} \right\}} 
            \varphi\left( \frac{x + \chi}{2}, \frac{t + \tau}{2} \right) 
            \delta_n \left( \frac{x - \chi}{2} \right) \delta_n \left( \frac{t - \tau}{2} \right) \\[5pt]
            & + \gamma_n(x,t) \p_X \varphi\left( \frac{x + \chi}{2}, \frac{t + \tau}{2} \right) 
            \delta_n \left( \frac{x - \chi}{2} \right) \delta_n \left( \frac{t - \tau}{2} \right),
        \end{aligned}    
    \] 

    we obtain:
    \[
        \begin{aligned}
            & \iiiint |\rho(x,t) - \sigma(\chi, \tau)| (\p_t \phi_n + \p_\tau \phi_n) \d{x} \d{t} \d{\chi} \d{\tau} \\[5pt]
            & \limit{n}{+\infty} 
            - \int_0^{+\infty} \dot y(t) \biggl( \left|\rho(y(t)+,t) - \sigma(y(t)+,t) \right| - \left|\rho(y(t)-,t) - \sigma(y(t)-,t) \right| \biggr)
            \varphi(y(t),t) \d{t} \\[5pt]
            & + \int_{0}^{+\infty} \int_{\R} |\rho(x,t) - \sigma(x, t)| \p_T \varphi(x,t) \d{x} \d{t},
        \end{aligned}
    \]

    and
    \[
        \begin{aligned}
            & \iiiint \Phi(\rho(x,t), \sigma(\chi, \tau)) (\p_x \phi_n + \p_\chi \phi_n) \d{x} \d{t} \d{\chi} \d{\tau} \\[5pt]
            & \limit{n}{+\infty} 
            \int_0^{+\infty} \biggl( \Phi(y(t)+,t), \sigma(y(t)+,t) - \Phi(\rho(y(t)-,t), \sigma(y(t)-,t)) \biggr)
            \varphi(y(t),t) \d{t} \\[5pt]
            & + \int_{0}^{+\infty} \int_{\R} \Phi(\rho(x,t), \sigma(x, t)) \p_X \varphi(x,t) \d{x} \d{t}.
        \end{aligned}
    \]

    Finally, since
    \[
        \begin{aligned}
            & \iiint |\rho_o(x) - \sigma(\chi, \tau)| \phi_n(x, 0, \chi, \tau) \d{x} \d{\chi} \d{\tau} \;\; \text{and} \;\; 
            \iiint |\rho(x,t) - \sigma_o(\chi)| \phi_n(x, t, \chi, 0) \d{x} \d{\chi} \d{t} \\[5pt]
            & \text{both converge to} \; \frac{1}{2} \int_{\R} |\rho_o(x) - \sigma_o(x)| \varphi(x,0) \d{x},
        \end{aligned}
    \]

    we get \eqref{1Bus_Kato1} by assembling the above ingredients together.
\end{proof}

\begin{theorem}
    \label{1Bus_stability_th1}
    Fix $\rho_o, \sigma_o \in \L{\infty}(\R, [0,1])$, 
    $y \in \Wloc{1}{\infty}(\mathopen]0, +\infty\mathclose[, \R)$, $q \in \Lloc{\infty}(\mathopen]0, +\infty\mathclose[, \R)$. We denote by $\rho$ a $\cG_{\dot y}(q)$-entropy solution to \eqref{1Bus}. The same way, let $\sigma$ be $\cG_{\dot y}(r)$-entropy solution to Problem \eqref{1Bus} with initial datum $\sigma_o$.
    We suppose that $q, r$ satisfy \eqref{1Bus_assumption_constraint}. Then for all $T > 0$, we have 
    \begin{equation}
        \label{1Bus_stability1}
        \|\rho(T) - \sigma(T)\|_{\L{1}(\R)} \leq \|\rho_o - \sigma_o\|_{\L{1}(\R)} 
        + 2 \int_0^{T} |q(t) - r(t)| \d{t}.
    \end{equation}

    In particular, Problem \eqref{1Bus} admits at most one solution.
\end{theorem}

\begin{proof}
    Fix $T > 0$, $R \geq \|y\|_{\L{\infty}(\mathopen]0, T\mathclose[)}$ and set 
    $L := \|f'\|_{\L{\infty}} + \|\dot y\|_{\L{\infty}(\mathopen]0, T\mathclose[)}$. Consider for all $n \in \N^*$ the function: 
    \[
        \varphi_n(x,t) := \frac{1}{4} \left(1 - \xi_n(t - T) \right) \left(1 - \xi_n\left(|x| - R + L(t - T) \right) \right),
    \]

    where $(\xi_n)_n$ is a smooth approximation of the sign function. The sequence $(\varphi_n)_n$ is a smooth approximation of the characteristic function of the 
    trapezoid
    \[
        \mathcal{T} := \left\{ (x, t) \in \overline{\Omega} \; : \; t \in [0, T] \; \text{and} \; |x| \leq R - L(t - T) \right\} 
        \supset \left\{ (y(t), t) \; : \; t \in [0, T] \right\}.    
    \]
    
    Let us apply Kato inequality \eqref{1Bus_Kato1} with $(\varphi_n)_n$. For all $n \in \N$, we have
    \[
        \begin{aligned}
            \int_{0}^{+\infty} \int_{\R} |\rho- \sigma| \p_{t}\varphi_n \d{x} \d{t}
            & = - \frac{1}{4} \int_{0}^{+\infty} \int_{\R} |\rho- \sigma| \xi_n' (t - T) 
            \left(1 - \xi_n\left(|x| - R + L(t - T) \right) \right) \d{x} \d{t} \\[5pt]
            & - \frac{L}{4} \int_{0}^{+\infty} \int_{\R} |\rho- \sigma| \left(1 - \xi_n(t - T) \right) 
            \xi_n' \left(|x| - R + L(t - T) \right) \d{x} \d{t} \\[5pt]
            & \limit{n}{+\infty} 
            - \int_{|x| \leq R} |\rho(x,T) -  \sigma(x, T)| \d{x} - L \int_{0}^{T} \int_{|x| = R - L(t-T)} |\rho- \sigma| \d{x} \d{t}.
        \end{aligned}
    \]

    Then,
    \[
        \begin{aligned}
            \int_{0}^{+\infty} \int_{\R} \Phi(\rho, \sigma) \p_{x}\varphi_n \d{x} \d{t} 
            & = - \frac{1}{4} \int_{0}^{+\infty} \int_{\R} \Phi(\rho, \sigma) \left(1 - \xi_n(t - T) \right) 
            \sgn(x) \xi_n' \left(|x| - R + L(t - T) \right) \d{x} \d{t} \\[5pt]
            & \limit{n}{+\infty} - \int_{0}^{T} \int_{|x| = R - L(t-T)} \Phi(\rho, \sigma) \sgn(x) \d{x} \d{t}.
        \end{aligned}
    \]

    Finally, we have
    \[
        \int_{\R} |\rho_o(x) - \sigma_o(x)| \varphi_n(x,0) \d{x}
        \limit{n}{+\infty}
        \int_{|x| \leq R + L T} |\rho_o(x) - \sigma_o(x)| \d{x}
    \]

    Remark also that the choices of $R$ and $L$ imply that for all $t > 0$,
    \[
        \varphi_n(y(t), t) \limit{n}{+\infty} 1.
    \]

    Assembling the previous limits together, we get:
    \[
        \begin{aligned}
            & - \int_{|x| \leq R} |\rho(x,T) -  \sigma(x, T)| \d{x} + \int_{|x| \leq R + L T} |\rho_o(x) - \sigma_o(x)| \d{x} \\[5pt]
            & - \int_{0}^{T} \int_{|x| = R - L(t-T)} \left( L |\rho- \sigma| + \Phi(\rho, \sigma) \sgn(x) \right) \d{x} \d{t} \\[5pt]
            & + \int_0^{T} \biggl( \Phi_{\dot y(t)} \left(\rho(y(t)+,t), \sigma(y(t)+,t) \right) - 
            \Phi_{\dot y(t)} \left(\rho(y(t)-,t), \sigma(y(t)-,t) \right) \biggr) \d{t} \geq 0.
        \end{aligned}
    \]

    Note that for all $\rho, \sigma \in [0, 1]$ and for all $x \in \R$,
    \[
        L |\rho - \sigma| + \Phi(\rho, \sigma) \sgn(x)  
        \geq L |\rho - \sigma| - |f(\rho) - f(\sigma)| 
        \geq (L - \|f'\|_{\L{\infty}}) |\rho - \sigma| \geq 0.
    \]

    Consequently, we have shown that
    \[
        \begin{aligned}
            \int_{|x| \leq R} |\rho(x,T) -  \sigma(x, T)| \d{x} 
            & \leq \int_{|x| \leq R + L T} |\rho_o(x) - \sigma_o(x)| \d{x} \\[5pt]
            & + \int_0^{T} \biggl( \underbrace{\Phi_{\dot y(t)} \left(\rho(y(t)+,t), \sigma(y(t)+,t) \right) - 
            \Phi_{\dot y(t)} \left(\rho(y(t)-,t), \sigma(y(t)-,t) \right)}_{\Delta(t)} \biggr) \d{t}.
        \end{aligned}
    \]

    What is left to do is to take the limit when $R \to +\infty$ and to estimate the last two terms of the right-hand side of the previous inequality. 
    The following table, in which we dropped the $t$-indexing, summarizes which values can take the difference $\Delta(t)$ according to which parts of 
    their respective germs the couples $(\rho(y(t)-,t), \rho(y(t)+,t))$ and $(\sigma(y(t)-,t), \sigma(y(t)+,t))$, respectively denoted by $(\rho_l, \rho_r)$ and 
    $(\sigma_l, \sigma_r)$ belong to.

    \begin{center}
        \renewcommand{\arraystretch}{1.5}
        \begin{tabular}{|c||c|c|c|}\hline
            \diagbox[width = 10em]{$(\sigma_l, \sigma_r)$}{$(\rho_l, \rho_r)$}
            & $\in \cG^1_{\dot y}(q)$ & $\in \cG^2_{\dot y}(q)$ & $\in \cG^3_{\dot y}(q)$ \\[5pt]
            \hline
            $\in \cG^1_{\dot y}(r)$ & $2 (q-r)$ & $0$ or $2(F_{\dot y}(\rho_l) - r)$ & $2 (F_{\dot y}(\rho_l) - r)$ \\[5pt]
            \hline
            $\in \cG^2_{\dot y}(r)$ & $0$ & $0$ & $ \leq 0$ \\[5pt]
            \hline
            $\in \cG^3_{\dot y}(r)$ & $2(F_{\dot y}(\sigma_l) - q)$ & $\leq 0$ & $\leq 0$ \\[5pt]
            \hline
        \end{tabular}
    \end{center}

    We clearly see the bound $\ds{\Delta(t) \leq 2|q(t) - r(t)|}$, which leads us to \eqref{1Bus_stability1}, which clearly implies uniqueness. This 
    concludes the proof.
\end{proof}

\section{Existence for the single trajectory problem}
\label{1Bus_Section_existence}

We build a simple finite volume scheme and prove its convergence to an admissible entropy solution to \eqref{1Bus}. From now on, we denote by
\[
	a \vee b := \max\{a,b\} \quad \text{and} \quad a \wedge b := \min\{a,b\}.
\]

Fix $\rho_o \in \L{\infty}(\R, [0,1])$ and $y \in \Wloc{1}{\infty}(\mathopen]0, +\infty\mathclose[, \R)$.

\subsection{Adapted mesh and definition of the scheme}
\label{1Bus_Section_mesh_scheme}

We start by defining the sequence of approximate slopes: 
\[
    \forall n \in \N, \; s^{n} = \frac{1}{\Delta t} \int_{t^n}^{t^{n+1}} \dot y(t) \d{t}; \quad 
    \forall t \geq 0, \; s_\Delta(t) = \sum_{n \in \N} s^n \1_{[t^n, t^{n+1} \mathclose[}(t),
\]

and the sequence of approximate trajectories:
\[
    \forall t \geq 0, \; y_\Delta(t) = y(0) 
    + \int_0^t s_{\Delta} (\tau) \d{\tau}; \quad 
    \forall n \in \N, \; y^n = y_\Delta(t^n).
\]

Since $(s_\Delta)_\Delta$ converges $\dot y$ in $\Lloc{1}(\mathopen]0, +\infty\mathclose[, \R)$, $(y_\Delta)_\Delta$ converges to $y$ in $\Lloc{\infty}(\mathopen]0, +\infty\mathclose[, \R)$.

The same way, we define $(q_\Delta)_\Delta$, the sequence of approximate constraints:
\[
    q_\Delta(t) = \sum_{n \in \N} q^n \1_{[t^n, t^{n+1} \mathclose[}(t); \quad q^n = \frac{1}{\Delta t} \int_{t^n}^{t^{n+1}} q(t) \d{t},
\]

which converges to $q$ in $\Lloc{1}(\mathopen]0, +\infty\mathclose[, \R)$.

\begin{remark}
    With our choices, from \eqref{1Bus_assumption_constraint2}, we deduce that
    \begin{equation}
        \label{1Bus_assumption_constraint3}
        \forall n \in \N, \quad s^n + q^n = \frac{1}{\Delta t} \int_{t^n}^{t^{n+1}} \left( \dot y(t) + q(t) \right) \d{t} \geq 0.   
    \end{equation}

    This fact will come in handy in the proof of stability for the scheme.
\end{remark}

Fix now $T > 0$ and a spatial mesh size $\Delta x > 0$ with $\lambda = \Delta t / \Delta x$ fixed, verifying the CFL condition
\begin{equation}
	\label{1Bus_cfl}
    2 \left(\underbrace{\|f'\|_{\L{\infty}} + \|\dot y \|_{\L{\infty}(\mathopen]0, T\mathclose[)}}_{:= L} \right) \lambda \leq 1.
\end{equation}

For all $n \in \N$, there exists a unique index $j_n \in \Z$ such that 
$y^n \in \mathopen]x_{j_n}, x_{j_n + 1}\mathclose[$, see Figure \ref{fig2}. Introduce 
the sequence $(\chi_j^n)_{j \in \Z}$ defined by 
\[
    \chi_j^n = 
    \left\{
        \begin{array}{ccl}
            x_j & \text{if} & j \leq j_n - 1 \\[5pt]
            y^n & \text{if} & j = j_n \\[5pt]
            x_{j+1} & \text{if} & j \geq j_n + 1.
        \end{array}
    \right.
\]

We define the cell grids:
\[
    \overline{\Omega} = \bigcup_{n \in \N} \bigcup_{j \in \Z} \cP_{j+1/2}^n,
\]

where for all $n \in \N$ and $j \in \Z$, $\cP_{j+1/2}^n$ is the rectangle 
$\mathopen]\chi_j^n, \chi_{j+1}^n\mathclose[ \times [t^n, t^{n+1} \mathclose[$ if $j \leq j_{n}-2$, one of the 
parallelograms represented in Figure \ref{fig2} if $j \in \{j_n - 1, j_n\}$ and the rectangle $\mathopen]\chi_{j+1}^n, \chi_{j+2}^n\mathclose[ \times [t^n, t^{n+1} \mathclose[$ if $j \geq j_{n} + 1$. 

\begin{figure}[!htp]
	\begin{center}
		\includegraphics[scale = 0.70]{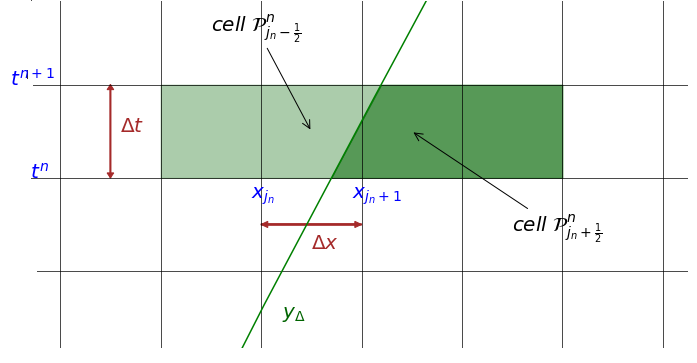}
	\end{center}
    \vspace*{-5mm}
	\caption{Illustration of the modification to the mesh.}
	\label{fig2}
\end{figure}

We start by discretizing the initial datum $\rho_o$ with $\left( \rho_{j+1/2}^{0} \right)_{j}$ where for all $j \in \Z$, $\ds{\rho_{j+1/2}^{0}}$ is its mean 
value on the cell $\mathopen]\chi_j^0, \chi_{j+1}^0 \mathclose[$. Clearly, for this choice, we have:
\[
    \rho_{j+1/2}^0 \in [0, 1]
    \quad \text{and} \quad 
    \rho_\Delta^0 = \sum_{j \in \Z} \rho_{j+1/2}^0 
    \1_{\mathopen]\chi_j^0, \chi_{j+1}^0 \mathclose[} \limit{\Delta x}{0} \rho_o \; \text{in} \; \Lloc{1}(\R).
\]

Let us denote by $\eo = \eo(a, b)$ the Engquist-Osher numerical flux associated with $f$ and for all $s \in \R$, $\god^s = \god^s(u, v)$ be the Godunov flux 
associated with $\rho \mapsto f(\rho) - s \rho$.

Fix $n \in \N$. To simplify the reading, we introduce the notations:
\begin{equation}
    \label{1Bus_numerical_fluxes}
    \forall j \in \Z, \quad f_j^n := \eo\left( \rho_{j - 1/2}^{n}, \rho_{j + 1/2}^{n} \right) \quad \text{and} \quad 
    f_{int}^n := \god^{s^n} \left( \rho_{j_n - 1/2}^{n}, \rho_{j_n + 1/2}^{n} \right) \wedge q^n.
\end{equation}

We now proceed to the definition of the scheme. It comes from a discretization of the conservation law written in each volume control 
$\cP_{j+1/2}^n$ ($n \in \N$, $j \in \Z$). Away from the trajectory/constraint, it is the standard $3$-point marching formula and when $j \in \{j_n-1, j_n\}$, we 
have to deal with both the constraint and the interface which is not vertical. Three cases have to be considered when describing the marching formula of the 
scheme, but we really give the details for only one of them.

\textbf{Case 1:} $j_{n+1} = j_n + 1$. This means that the line joining $(y^n, t^n)$ and $(y^{n+1}, t^{n+1})$ 
crosses the line $x = x_{j_{n} + 1}$, see Figure \ref{fig2}. If $j \notin \{j_n - 1, j_n \}$, the conservation written in the rectangle 
$\cP_{j+1/2}^n$ is given by the standard equation:
\begin{equation}
    \label{1Bus_mf1}
    \left( \rho_{j+1/2}^{n+1} - \rho_{j+1/2}^{n} \right) \Delta x + (f_{j+1}^n - f_j^n) \Delta t = 0.
\end{equation}

From the conservation in the cell $\cP_{j_n - 1/2}^n$, we set:
\begin{equation}
    \label{1Bus_mf2}
    \rho_{j_{n+1} - 1/2}^{n+1} \left( y^{n+1} - \chi_{j_{n+1}-2}^{n+1} \right) 
    - \rho_{j_n - 1/2}^{n} \left( y^{n} - \chi_{j_{n}-1}^{n} \right) + (f_{int}^n - f_{j_{n} - 1}^n) \Delta t = 0.
\end{equation}

This formula corresponds to the choice of putting the same value for $\rho_\Delta$ on 
$\mathopen]\chi_{j_{n+1}-2}^{n+1}, \chi_{j_{n+1}-1}^{n+1}\mathclose[$ and on 
$\mathopen]\chi_{j_{n+1}-1}^{n+1}, y^{n+1}\mathclose[$ at time $t=t^{n+1}$, \textit{i.e.} $\rho_{j_{n+1} - 3/2}^{n+1} = \rho_{j_{n+1} - 1/2}^{n+1}$. In the cell $\cP_{j_n + 1/2}^n$, the 
conservation takes the form: 

\begin{equation}
    \label{1Bus_mf3}
    \rho_{j_{n+1} + 1/2}^{n+1} \left(\chi_{j_{n+1} + 1}^{n+1} - y^{n+1} \right) 
    - \rho_{j_n + 1/2}^{n} \left( \chi_{j_{n} + 1}^{n} - y^n \right) - \rho_{j_n + 3/2}^{n} \Delta x + (f_{j_n + 2 }^n - f_{int}^n) \Delta t = 0.
\end{equation}

Let us introduce the two functions 
\[
    H_{j_n - 1}^n (u, v, w) 
    := \frac{v (y^n - \chi_{j_n - 1}^{n}) - \left( \god^{s^n}(v, w) \wedge q^n - \eo(u, v)\right) \Delta t}{y^{n+1} - \chi_{j_{n+1} - 2}^{n+1}}  
\]

and 
\[
    H_{j_n}^n (u, v, w, z) 
    := \frac{v(\chi_{j_n + 1}^n - y^n) + w \Delta x - \left(\eo(w, z) - \god^{s^n}(u, v) \wedge q^n \right)\Delta t }{\chi_{j_{n+1} + 1}^{n+1} - y^{n+1}},  
\]

so that 
\begin{equation}
    \label{1Bus_mf4}
    \left\{
        \begin{aligned}
            \rho_{j_{n+1}- 1/2}^{n+1} & = H_{j_n - 1}^n (\rho_{j_{n}- 3/2}^{n}, \rho_{j_{n} - 1/2}^{n}, \rho_{j_{n} + 1/2}^{n}) \\[5pt]
            \rho_{j_{n+1} + 1/2}^{n+1} & = H_{j_n}^n (\rho_{j_{n}- 1/2}^{n}, \rho_{j_{n} + 1/2}^{n}, \rho_{j_{n} + 3/2}^{n}, \rho_{j_{n} + 5/2}^{n}).
        \end{aligned}
    \right.
\end{equation}

The key point in the proofs of the next section (stability and discrete entropy inequalities) is that the functions $H_{j_n-1}^n$ and $H_{j_n}^n$ are 
nondecreasing with respect to their arguments, therefore the modification in \eqref{1Bus_numerical_fluxes} did not affect the monotonicity of the 
resulting scheme \eqref{1Bus_mf1}~--~\eqref{1Bus_mf3}.

Finally, the approximate solution $\rho_\Delta$ is defined almost everywhere on $\overline{\Omega}$:
\[
	\rho_{\Delta} = \sum_{n \in \N} \left( 
        \sum_{j \leq j_n} \rho_{j+1/2}^{n} \1_{\cP_{j+1/2}^{n}} + \sum_{j \geq j_n + 1} \rho_{j+3/2}^{n} \1_{\cP_{j+1/2}^{n}} \right).
\]

The other cases ($j_{n+1} = j_n$ or $j_{n+1} = j_{n} - 1$) follow from similar geometric considerations. Note that in the context of traffic dynamics, $y$ 
would be the trajectory of a stationary or a forward moving obstacle and therefore, we should have $\dot y \geq 0$. This implies that for all $n \in \N$, 
either $j_{n+1} = j_n$ or $j_{n+1} = j_n + 1$. This is why we will focus on the case presented in Figure \ref{fig2}.

\subsection{Stability and discrete entropy inequalities}

\begin{proposition}[$\L{\infty}$ stability]
    \label{1Bus_scheme_stability_pp}
	Under the CFL condition \eqref{1Bus_cfl}, the scheme \eqref{1Bus_mf1}~--~\eqref{1Bus_mf3} is stable: 
	\begin{equation}
		\label{1Bus_scheme_stability}
		\forall n \in \N, \; \forall j \in \Z, \quad \rho_{j+1/2}^{n} \in [0,1].
	\end{equation}
\end{proposition}

\begin{proof}
    \textit{Monotonicity.} Fix $n \in \N$. Clearly, the expression \eqref{1Bus_mf1} allows us to express $\rho^{n+1}$ as a function of three 
    values of $\rho^n$ in a nondrecreasing way, see the \cite[Chapter 5]{EGHBook} for instance. We now verify that the functions $H_{j_n - 1}^n$ and 
    $H_{j_n}^n$ are also nondecreasing. Let us detail the proof for $H_{j_n}^n$. Recall that $H_{j_n}^n$ is Lipschitz continuous by construction, therefore 
    we can study its monotonicity in terms of its a.e. derivatives. Making use of both the CFL condition \eqref{1Bus_cfl} and of the monotonicity of $\eo$ 
    and $\god^{s^n}$, for a.e. $u, v, w, z \in [0, 1]$, we have
    \[
        \begin{aligned}
            \p_1 H_{j_n}^n(u, v, w, z)
            & = \frac{1}{2} \frac{\Delta t}{\chi_{j_{n+1} + 1}^{n+1} - y^{n+1}} \frac{\p \god^{s^n}}{\p a}(u, v) 
            (1 - \sgn(\god^{s^n}(u, v) - q^n)) \geq 0, \\[5pt]
            \p_2 H_{j_n}^n(u, v, w, z) 
            & = \frac{\chi_{j_n + 1}^{n} - y^n}{\chi_{j_{n+1} + 1}^{n+1} - y^{n+1}}
            + \frac{\Delta t}{\chi_{j_{n+1} + 1}^{n+1} - y^{n+1}} \frac{\p \god^{s^n}}{\p b}(u, v) 
            \frac{(1 - \sgn(\god^{s^n}(u, v) - q^n))}{2} \\[5pt]
            & \geq \frac{\chi_{j_n + 1}^{n} - (y^n + L \Delta t)}{\chi_{j_{n+1} + 1}^{n+1} - y^{n+1}}
            \geq \frac{\chi_{j_n + 1}^{n} - \left(y^n + \frac{\Delta x}{2} \right)}{\chi_{j_{n+1} + 1}^{n+1} - y^{n+1}} \geq 0, \\[5pt]
            \p_3 H_{j_n}^n(u, v, w, z) 
            & = \frac{\Delta x}{\chi_{j_{n+1} + 1}^{n+1} - y^{n+1}} 
            - \frac{\Delta t}{\chi_{j_{n+1} + 1}^{n+1} - y^{n+1}} \frac{\p \eo}{\p a}(w, z) \\[5pt]
            & \geq \frac{\Delta x - L \Delta t}{\chi_{j_{n+1} + 1}^{n+1} - y^{n+1}}
            \geq \frac{\Delta x - \Delta x/2}{\chi_{j_{n+1} + 1}^{n+1} - y^{n+1}} \geq 0, \\[5pt]
            \p_4 H_{j_n}^n(u, v, w, z) 
            & = -\frac{\Delta t}{\chi_{j_{n+1} + 1}^{n+1} - y^{n+1}} \frac{\p \eo}{\p b}(w, z) \geq 0,
        \end{aligned}
    \]

    proving the monotonicity of $H_{j_n}^n$. Similar computations show that $H_{j_n - 1}^n$ is nondecreasing with respect to its arguments as well. \\
    \textit{Stability.} We now turn to the proof of \eqref{1Bus_scheme_stability}, which is done by induction on $n$. If $n=0$, it is verified by 
    definition of $\left( \rho_{j+1/2}^{0} \right)_{j}$. Suppose now that \eqref{1Bus_scheme_stability} holds for some integer $n \geq 0$ and let us show 
    that it still holds for $n+1$. Remark that $0$ and $1$ are stationary solutions to the scheme. It is obviously true in the case 
    \eqref{1Bus_mf1}. The definitions of $H_{j_n - 1}^n$ and $H_{j_n}^n$ do not change this fact. For instance, $H_{j_n - 1}^n (0, 0, 0) = 0$ since 
    $q^n \geq 0$ and because of \eqref{1Bus_assumption_constraint3}, we also have:
    \[ 
        H_{j_n - 1}^n (1, 1, 1) 
        = \frac{(y^n - \chi_{j_n - 1}^{n}) - \left( (-s^n) \wedge q^n \right) \Delta t}{y^{n+1} - \chi_{j_{n+1} - 2}^{n+1}}
        = \frac{(y^n - \chi_{j_n - 1}^{n}) + s^n \Delta t}{y^{n+1} - \chi_{j_{n+1} - 2}^{n+1}}
        = 1.
    \]

    Similar computations would ensure that it holds also for $H_{j_n}^n$. Using now the monotonicity of $H_{j_n - 1}^n$ for instance, we deduce that
    \[
        \begin{aligned}
            0 = H_{j_n - 1}^n (0, 0, 0) & 
            \leq H_{j_n - 1}^n (\rho_{j_n - 3/2}^n, \rho_{j_n - 1/2}^n, \rho_{j_n + 1/2}^n) \\
            & = \rho_{j_{n+1} - 1/2}^{n+1} \\
            & = H_{j_n - 1}^n (\rho_{j_n - 3/2}^n, \rho_{j_n - 1/2}^n, \rho_{j_n + 1/2}^n) \leq H_{j_n - 1}^n (1, 1, 1) = 1,
        \end{aligned}
    \]

    which concludes the induction argument. The remaining cases follow from similar computations.
\end{proof}

\begin{corollary}[Discrete entropy inequalities]
    \label{1Bus_dei_cor}
    Fix $n \in \N$, $j \in \Z \backslash \{j_{n+1} - 2\}$ and $\kappa \in [0, 1]$. Then the numerical scheme 
    \eqref{1Bus_mf1}~--~\eqref{1Bus_mf3} fulfills the following discrete entropy inequalities:
    \begin{equation}
        \label{1Bus_dei}
        \begin{aligned}
            |\rho_{j+1/2}^{n+1}-\kappa| (\chi_{j+1}^{n+1} - \chi_{j}^{n+1}) \leq 
            \left\{
                \begin{array}{lll}
                    & |\rho_{j+1/2}^{n}-\kappa| (\chi_{j+1}^{n} - \chi_{j}^{n}) - \left( \Phi_{j+1}^{n} - \Phi_{j}^{n} \right) \Delta t 
                    & \text{if} \; j \notin \{j_{n+1} - 1, j_{n+1} \} \\[10pt]
                    & - |\rho_{j_{n+1} - 1/2}^{n+1}-\kappa| \Delta x + |\rho_{j_n - 1/2}^{n}-\kappa| (\chi_{j_n}^{n} - \chi_{j_n - 1}^{n}) & \\[5pt]
                    & - \left( \Phi_{int}^{n} - \Phi_{j_n - 1}^{n} \right) \Delta t + \frac{1}{2} \cR_{s^n}(\kappa, q^n) \Delta t 
                    & \text{if} \; j = j_{n+1} - 1 \\[10pt]
                    & |\rho_{j_n + 1/2}^{n}-\kappa| (\chi_{j_n + 1}^{n} - \chi_{j_n}^{n}) + |\rho_{j_n + 3/2}^{n}-\kappa| \Delta x \\[5pt]
                    & - \left( \Phi_{j_n + 2}^n - \Phi_{int}^{n} \right) \Delta t + \frac{1}{2} \cR_{s^n}(\kappa, q^n) \Delta t 
                    & \text{if} \; j = j_{n+1},
                \end{array}
            \right.
        \end{aligned}
    \end{equation}

    where $\Phi_{j}^{n}$ and $\Phi_{int}^{n}$ denote the numerical entropy fluxes:
	\[
		\begin{aligned}
			& \Phi_{j}^{n} := \eo(\rho_{j-1/2}^{n} \vee \kappa,\rho_{j+1/2}^{n} \vee \kappa) 
            - \eo(\rho_{j-1/2}^{n} \wedge \kappa,\rho_{j+1/2}^{n} \wedge \kappa); \\[5pt]
			& \Phi_{int}^{n} := \min\{\god^{s^n}(\rho_{j_{n} - 1/2}^{n} \vee \kappa,\rho_{j_{n} + 1/2}^{n} \vee \kappa),q^{n}\}
			- \min\{\god^{s^n}(\rho_{j_{n} - 1/2}^{n} \wedge \kappa,\rho_{j_{n} + 1/2}^{n} \wedge \kappa), q^{n}\}
		\end{aligned}
	\]
\end{corollary}

\begin{proof}
    This result is mostly a consequence of the scheme monotonicity. When the interface/constraint does not enter the calculations \textit{i.e.} when 
    $j \notin \{j_{n+1} - 1, j_{n+1} \}$, the proof follows \cite[Lemma 5.4]{EGHBook}. The key point is not only the monotonicity, but also the fact that in the 
    classical case, all the constants states $\kappa \in [0,1]$ are stationary solutions of the scheme. This observation does not hold when the constraint enters 
    the calculations. Suppose for example that $j = j_{n+1}$ (which corresponds to the function $H_{j_n}^n$). Here, we have
	\[
        \begin{aligned}
            H_{j_{n}}^{n}(\kappa,\kappa,\kappa, \kappa) 
            & = \frac{\kappa (\chi_{j_n + 1}^n - y^n) + \kappa \Delta x - 
            \left(f(\kappa) - (f(\kappa) -s^n \kappa)\wedge q^n \right) \Delta t}{\chi_{j_{n+1} + 1}^{n+1} - y^{n+1}} \\[5pt] 
            & = \frac{(\chi_{j_n + 2}^n - y^n -s^n \Delta t) \kappa}{\chi_{j_{n+1} + 1}^{n+1} - y^{n+1}} 
            - \frac{\Delta t}{2 (\chi_{j_{n+1} + 1}^{n+1} - y^{n+1})} \cR_{s^n}(\kappa, q^n) \\[5pt]
            & = \kappa - \frac{\Delta t}{2 (\chi_{j_{n+1} + 1}^{n+1} - y^{n+1})} \cR_{s^n}(\kappa, q^n),
        \end{aligned}
    \]
    
    and it implies:
    \[
        \begin{aligned}
            & H_{j_{n}}^{n}(\rho_{j_{n} - 1/2}^n \wedge \kappa, \rho_{j_{n} + 1/2}^n \wedge \kappa, \rho_{j_{n} + 3/2}^n \wedge \kappa, 
            \rho_{j_{n} + 5/2}^n \wedge \kappa) \\[5pt]
            & \leq \rho_{j_{n+1} + 1/2}^{n+1} \wedge \kappa, \ \rho_{j_{n+1} + 1/2}^{n+1} \vee \kappa \\[5pt]
            & \leq H_{j_{n}}^{n}(\rho_{j_{n} - 1/2}^n \vee \kappa, \rho_{j_{n} + 1/2}^n \vee \kappa, \rho_{j_{n} + 3/2}^n \vee \kappa, 
            \rho_{j_{n} + 5/2}^n \vee \kappa) + \frac{\Delta t}{2 (\chi_{j_{n+1} + 1}^{n+1} - y^{n+1})} \cR_{s^n}(\kappa, q^n).
        \end{aligned}
    \]

    We deduce: 
    \[
		\begin{aligned}
            |\rho_{j_{n+1} + 1/2}^{n+1} - \kappa|
            & = \rho_{j_{n+1} + 1/2}^{n+1} \vee \kappa - \rho_{j_{n+1} + 1/2}^{n+1} \wedge \kappa \\[5pt]
            & \leq H_{j_{n}}^{n}(\rho_{j_{n} - 1/2}^n \vee \kappa, \rho_{j_{n} + 1/2}^n \vee \kappa, \rho_{j_{n} + 3/2}^n \vee \kappa,
            \rho_{j_{n} + 5/2}^n \vee \kappa) \\[5pt]
            & - H_{j_{n}}^{n}(\rho_{j_{n} - 1/2}^n \wedge \kappa, \rho_{j_{n} + 1/2}^n \wedge \kappa, \rho_{j_{n} + 3/2}^n \wedge \kappa, 
            \rho_{j_{n} + 5/2}^n \wedge \kappa)
            + \frac{\Delta t}{2 (\chi_{j_{n+1} + 1}^{n+1} - y^{n+1})} \cR_{s^n}(\kappa, q^n) \\[5pt]
            & = \frac{\chi_{j_n + 1}^n - y^n}{\chi_{j_{n+1} + 1}^{n+1} - y^{n+1}} |\rho_{j_{n} + 1/2}^{n} - \kappa| 
            + \frac{\Delta x}{\chi_{j_{n+1} + 1}^{n+1} - y^{n+1}} |\rho_{j_{n} + 3/2}^{n} - \kappa| \\[5pt]
            & - \frac{\Delta t}{\chi_{j_{n+1} + 1}^{n+1} - y^{n+1}} \left( \Phi_{j_{n} + 2}^n - \Phi_{int}^{n} \right) 
            + \frac{\Delta t}{2 (\chi_{j_{n+1} + 1}^{n+1} - y^{n+1})} \cR_{s^n}(\kappa, q^n),
		\end{aligned}
	\]

    which is exactly \eqref{1Bus_dei} in the case $j= j_{n+1}$. The obtaining of \eqref{1Bus_dei} in the case $j = j_{n+1} - 1$ is similar, so we 
    omit the details of the proof for this case.
\end{proof}

\subsection{Continuous inequalities for the approximate solution}

The next step of the reasoning is to derive analogous inequalities to \eqref{1Bus_ei1}-\eqref{1Bus_ci1}, verified by the approximate 
solution $\rho_\Delta$, starting from \eqref{1Bus_dei} and 
\eqref{1Bus_mf1}~--~\eqref{1Bus_mf3}.

In this section, we fix a test function $\varphi \in \Cc{\infty}(\overline{\Omega}, \R^+)$ and set:
\[
    \forall n \in \N, \; \forall j \in \Z, \quad 
    \varphi_{j+1/2}^{n} 
    := \frac{1}{\chi_{j+1}^n - \chi_{j}^n} \int_{\chi_{j}^n}^{\chi_{j+1}^n} \varphi(x, t^n) \d{x}
    = \fint _{\chi_{j}^n}^{\chi_{j+1}^n} \varphi(x, t^n) \d{x}.
\]

We start by deriving continuous entropy inequalities verified by $\rho_\Delta$. Define the approximate entropy flux:
\[
	\Phi_\Delta(\rho_\Delta, \kappa) := \sum_{n \in \N} \left( 
        \sum_{j \leq j_n} \Phi_{j}^n \1_{\cP_{j+1/2}^{n}} + \sum_{j \geq j_n + 1} \Phi_{j+1}^n \1_{\cP_{j+1/2}^{n}} \right).
\]

\begin{proposition}[Approximate entropy inequalities]
    \label{1Bus_aei_pp}
    Fix $n \in \N$ and $\kappa \in [0,1]$. Then we have 
	\begin{equation}
		\label{1Bus_aei1}
		\begin{aligned}
            & \int_{t^n}^{t^{n+1}} \int_{\R} \biggl( |\rho_{\Delta}-\kappa| \p_{t} \varphi 
            + \Phi_\Delta \left( \rho_{\Delta},\kappa \right) \p_{x} \varphi \biggr) \d{x} \d{t} \\[5pt]
            & + \int_{\R} |\rho_{\Delta}(x, t^n)-\kappa| \varphi(x, t^n) \d{x} - \int_{\R} |\rho_{\Delta}(x, t^{n+1})-\kappa| \varphi(x, t^{n+1}) \d{x} \\[5pt]
            & + \int_{t^n}^{t^{n+1}} \cR_{s_\Delta (t)} (\kappa, q_{\Delta}(t)) \varphi(y_\Delta(t), t) \d{t} 
            \geq \bigO{\Delta x^2} + \bigO{\Delta x \Delta t} + \bigO{\Delta t^2}.
		\end{aligned}
    \end{equation}
\end{proposition}

\begin{proof}
    For all $j \in \Z \backslash \{j_{n+1} - 2\}$, we multiply the discrete entropy inequalities \eqref{1Bus_dei} by $\varphi_{j+1/2}^{n+1}$ and take the sum to 
    obtain:
    \[
        \begin{aligned}
            & \sum_{j \neq j_{n+1} - 2} \left| \rho_{j+1/2}^{n+1} - \kappa \right| 
            \varphi_{j+1/2}^{n+1} (\chi_{j+1}^{n+1} - \chi_{j}^{n+1}) \\[5pt]
            & \leq \sum_{j \notin \{j_{n+1} - 2, j_{n+1} - 1, j_{n+1} \}} \left( \left| \rho_{j+1/2}^{n} - \kappa \right| 
            (\chi_{j+1}^{n} - \chi_{j}^{n}) - (\Phi_{j+1}^n - \Phi_j^n) \Delta t \right) \varphi_{j+1/2}^{n+1} \\[5pt]
            & + |\rho_{j_n - 1/2}^{n}-\kappa| \varphi_{j_{n+1} - 1/2}^{n+1} (\chi_{j_n}^{n} - \chi_{j_n - 1}^{n}) 
            - |\rho_{j_{n+1} - 1/2}^{n+1}-\kappa| \varphi_{j_{n+1} - 1/2}^{n+1} \Delta x
            - \left( \Phi_{int}^{n} - \Phi_{j_n - 1}^{n} \right) \varphi_{j_{n+1} - 1/2}^{n+1} \Delta t \\[5pt]
            & + |\rho_{j_n + 1/2}^{n}-\kappa| \varphi_{j_{n+1} + 1/2}^{n+1} (\chi_{j_n + 1}^{n} - \chi_{j_n}^{n}) 
            + |\rho_{j_n + 3/2}^{n}-\kappa| \varphi_{j_{n+1} + 1/2}^{n+1} \Delta x 
            - \left( \Phi_{j_n + 2}^n - \Phi_{int}^{n} \right) \varphi_{j_{n+1} + 1/2}^{n+1} \Delta t \\[5pt]
            & + \frac{1}{2} \cR_{s^n}(\kappa, q^n)  (\varphi_{j_{n+1} - 1/2}^{n+1} + \varphi_{j_{n+1} + 1/2}^{n+1}) \Delta t.
        \end{aligned}
    \]

    This inequality can be rewritten as
    \[
        \begin{aligned}
            & \sum_{j \in \Z} \left| \rho_{j+1/2}^{n+1} - \kappa \right| \varphi_{j+1/2}^{n+1} (\chi_{j+1}^{n+1} - \chi_{j}^{n+1})
            - \sum_{j \in \Z} \left| \rho_{j+1/2}^{n} - \kappa \right| \varphi_{j+1/2}^{n+1} (\chi_{j+1}^{n} - \chi_{j}^{n}) \\[5pt]
            & \leq - \underbrace{\left| \rho_{j_{n+1} - 1/2}^{n+1} - \kappa \right| 
            \left( \varphi_{j_{n+1} - 1/2}^{n+1} - \varphi_{j_{n+1} - 3/2}^{n+1} \right) \Delta x}_{\eps_1}
            + \underbrace{\left| \rho_{j_n - 1/2}^{n} - \kappa \right| 
            \left( \varphi_{j_{n+1} - 1/2}^{n+1} - \varphi_{j_{n+1} - 3/2}^{n+1} \right) (\chi_{j_n}^{n} - \chi_{j_n - 1}^{n})}_{\eps_2} \\[5pt]
            & + \underbrace{\left| \rho_{j_{n} + 1/2}^{n} - \kappa \right| 
            \left( \varphi_{j_{n+1} + 1/2}^{n+1} - \varphi_{j_{n+1} - 1/2}^{n+1} \right) (\chi_{j_n + 1}^{n} - \chi_{j_n}^{n})}_{\eps_3} \\[5pt]
            & - \sum_{j \notin \{j_{n+1}-2 , j_{n+1}-1, j_{n+1} \}} (\Phi_{j+1}^n - \Phi_j^n) \varphi_{j+1/2}^{n+1} \Delta t
            - \left( \Phi_{int}^{n} - \Phi_{j_n - 1}^{n} \right) \varphi_{j_{n+1} - 1/2}^{n+1} \Delta t \\[5pt]
            & - \left( \Phi_{j_n + 2}^n - \Phi_{int}^{n} \right) 
            \varphi_{j_{n+1} + 1/2}^{n+1} \Delta t 
            + \frac{1}{2} \cR_{s^n}(\kappa, q^n)  (\varphi_{j_{n+1} - 1/2}^{n+1} + \varphi_{j_{n+1} + 1/2}^{n+1}) \Delta t,
        \end{aligned}
    \]

    with for all $i \in [\![1; 3]\!]$, 
    $|\eps_i| \leq 8 \|\p_x \varphi \|_{\L{\infty}} \Delta x^2$.
    We now proceed to the Abel's transformation and reorganize the terms of the inequality. This leads us to:
    \[
        \begin{aligned}
            & \underbrace{\sum_{j \in \Z} \left| \rho_{j+1/2}^{n+1} - \kappa \right| \varphi_{j+1/2}^{n+1} (\chi_{j+1}^{n+1} - \chi_{j}^{n+1}) 
            - \sum_{j \in \Z} \left| \rho_{j+1/2}^{n} - \kappa \right| \varphi_{j+1/2}^{n} (\chi_{j+1}^{n} - \chi_{j}^{n})}_{A} \\[5pt]
            & - \underbrace{\sum_{j \in \Z} \left| \rho_{j+1/2}^{n} - \kappa \right| \left( \varphi_{j+1/2}^{n+1} - \varphi_{j+1/2}^{n} \right) 
            (\chi_{j+1}^{n} - \chi_{j}^{n})}_{B} 
            + \underbrace{\sum_{j \notin \{j_{n+1} - 2, j_{n + 1} - 1 \}} \Phi_j^n 
            \left( \varphi_{j+1/2}^{n+1} - \varphi_{j-1/2}^{n+1} \right) \Delta t}_{C} \\[5pt]
            & \leq \underbrace{\frac{1}{2} \cR_{s^n}(\kappa, q^n)  (\varphi_{j_{n+1} - 1/2}^{n+1} + \varphi_{j_{n+1} + 1/2}^{n+1}) \Delta t}_{D} 
            + \sum_{i=1}^5 \eps_i,
        \end{aligned}
    \]

    with for all $i \in [\![4; 5]\!]$, 
    $|\eps_i| \leq 4 \|f\|_{\L{\infty}} \|\p_x \varphi \|_{\L{\infty}} \Delta x \Delta t$. We immediately see that 
    \[
        A = \int_{\R} \left|\rho_{\Delta}(x, t^{n+1}) - \kappa \right| \varphi(x, t^{n+1}) \d{x}
        - \int_{\R} \left|\rho_{\Delta}(x, t^{n}) - \kappa \right| \varphi(x, t^{n}) \d{x}.
    \]
    
    We conclude this proof by estimating the remaining terms of the inequality.

    \textbf{Estimating $B$.} First, note that 
    \[
        \begin{aligned}
            B 
            & = \sum_{j \leq j_n - 2} \iint_{\cP_{j+1/2}^n} \left|\rho_\Delta - \kappa \right| \p_t \varphi \d{x} \d{t} 
            + \sum_{j \geq j_n + 1} \iint_{\cP_{j+1/2}^n} \left|\rho_\Delta - \kappa \right| \p_t \varphi \d{x} \d{t} \\[5pt]
            & + \underbrace{\left| \rho_{j_n - 1/2}^{n} - \kappa \right| \biggl( 
            \fint_{\chi_{j_n - 1}^{n+1}}^{\chi_{j_n + 1}^{n+1}} \varphi(x, t^{n+1}) \d{x} - \fint_{\chi_{j_n - 1}^{n}}^{y^{n}} \varphi(x, t^n) \d{x} \biggr)
            (y^n - \chi_{j_n - 1}^n)}_{B_1} \\[5pt]
            & + \underbrace{\left| \rho_{j_n + 1/2}^{n} - \kappa \right| \biggl( 
            \fint_{\chi_{j_n}^{n+1}}^{y^{n+1}} \varphi(x, t^{n+1}) \d{x} - \fint_{y^{n}}^{\chi_{j_n + 1}^n} \varphi(x, t^n) \d{x} \biggr)
            (\chi_{j_n + 1}^n - y^n)}_{B_2} \\[5pt]
            & + \underbrace{\left| \rho_{j_n + 3/2}^{n} - \kappa \right| \biggl( 
            \fint_{y^{n+1}}^{\chi_{j_n + 2}^{n+1}} \varphi(x, t^{n+1}) \d{x} - \fint_{\chi_{j_n + 1}^n}^{\chi_{j_n + 2}^n} \varphi(x, t^n) \d{x} \biggr)
            \Delta x}_{B_3}.
        \end{aligned}    
    \]

    Since
    \[
        \begin{aligned}
            & \iint_{\cP_{j_n - 1/2}^n} \left|\rho_\Delta - \kappa \right| \p_t \varphi \d{x} \d{t} \\[5pt]
            & = \left| \rho_{j_n - 1/2}^{n} - \kappa \right| \biggl( 
            \int_{\chi_{j_n - 1}^{n+1}}^{y^{n+1}} \varphi(x, t^{n+1}) \d{x} - \int_{\chi_{j_n-1}^{n}}^{y^{n}} \varphi(x, t^n) \d{x} 
            - s^n \int_{t^n}^{t^{n+1}} \varphi(y^n + s^n (t - t^n), t) \d{t} \biggr) \\[5pt]
            & = \left| \rho_{j_n - 1/2}^{n} - \kappa \right| 
            \biggl( \frac{y^{n+1} - \chi_{j_n - 1}^{n+1}}{y^n - \chi_{j_n - 1}^n} \fint_{\chi_{j_n - 1}^{n+1}}^{y^{n+1}} \varphi(x, t^{n+1}) \d{x} 
            - \fint_{\chi_{j_n-1}^{n}}^{y^{n}} \varphi(x, t^n) \d{x} \biggr. \\[5pt]
            & \biggl. 
            + \frac{y^n - y^{n+1}}{y^n - \chi_{j_n - 1}^n} \fint_{t^n}^{t^{n+1}} \varphi(y^n + s^n (t - t^n), t) \d{t} \biggr) (y^n - \chi_{j_n - 1}^n),
        \end{aligned}
    \]

    we deduce the bound:
    \[
        \begin{aligned}
            & \left| B_1 - \iint_{\cP_{j_n - 1/2}^n} \left|\rho_\Delta - \kappa \right| \p_t \varphi \d{x} \d{t} \right| \\[5pt]
            & = \left| \rho_{j_n - 1/2}^{n} - \kappa \right| (y^{n+1} - y^n) \left| 
            \fint_{\chi_{j_n - 1}^{n+1}}^{y^{n+1}} \varphi(x, t^{n+1}) \d{x} - \fint_{t^n}^{t^{n+1}} \varphi(y^n + s^n (t - t^n), t) \d{t}\right| \\[5pt]
            & \leq \|\dot y \|_{\L{\infty}} \biggl( 3 \|\p_x \varphi \|_{\L{\infty}} \Delta x + \|\p_t \varphi \|_{\L{\infty}} \Delta t  
            + 2 \|\dot y \|_{\L{\infty}} \|\p_x \varphi \|_{\L{\infty}} \Delta t \biggr) \Delta t.
        \end{aligned}
    \]
    
    The same way, we would derive the estimation:
    \[
        \begin{aligned}
            & \left| B_2 + B_3 - \iint_{\cP_{j_n + 1/2}^n} \left|\rho_\Delta - \kappa \right| \p_t \varphi \d{x} \d{t} \right| \\[5pt]
            & \leq 6 \|\p_x \varphi\|_{\L{\infty}} \Delta x^2 
            + \|\dot y \|_{\L{\infty}} \biggl( 2 \|\p_x \varphi \|_{\L{\infty}} \Delta x + \|\p_t \varphi \|_{\L{\infty}} \Delta t 
            + 2 \|\dot y \|_{\L{\infty}} \|\p_x \varphi \|_{\L{\infty}} \Delta t \biggr) \Delta t.
        \end{aligned}
    \]

    \textbf{Estimating $C$.} We write:
    \[
        \begin{aligned}
            C
            & = \lambda \sum_{j \notin \{j_{n+1}-2, j_{n+1}-1, j_{n+1}\}} 
            \int_{\chi_j^{n}}^{\chi_{j+1}^n} \int_{x-\Delta x}^x \Phi_j^n \p_x \varphi(y, t^{n+1}) \d{y} \d{x} 
            + \underbrace{\Phi_{j_{n+1}}^n \left( \varphi_{j_{n+1} + 1/2}^{n+1} - \varphi_{j_{n+1} - 1/2}^{n+1} \right) \Delta t}_{\eps_6} \\[5pt]
            & = \int_{t^n}^{t^{n+1}} \int_{\R} \Phi_\Delta(\rho_\Delta, \kappa) \p_x \varphi \d{x} \d{t} + \eps_6 
            - \underbrace{\sum_{j_{n+1}-2 \leq j \leq j_{n+1}-1} \iint_{\cP_{j+1/2}^n} \Phi_\Delta(\rho_\Delta, \kappa) \p_x \varphi \d{x} \d{t}}_{\eps_7} \\[5pt]
            & + \underbrace{\sum_{j \notin \{j_{n+1}-2, j_{n+1}-1, j_{n+1}\}} 
            \biggl( \lambda \int_{\chi_j^{n}}^{\chi_{j+1}^n} \int_{x-\Delta x}^x \Phi_j^n \p_x \varphi(y, t^{n+1}) \d{y} \d{x}  \biggr)
            - \int_{t^n}^{t^{n+1}} \int_{\R} \Phi_\Delta(\rho_\Delta, \kappa) \p_x \varphi \d{x} \d{t}}_{\eps_8},
        \end{aligned}
    \]

    with 
    \[
        |\eps_6| + |\eps_7| \leq 8 \|f\|_{\L{\infty}} \|\p_x \varphi \|_{\L{\infty}} \Delta x \Delta t,
    \]

    and 
    \[
        |\eps_8| \leq \|f\|_{\L{\infty}} \biggl( 
            4 \|\p_{xx}^2 \varphi\|_{\L{\infty}(\R^+, \L{1})} \Delta x 
        +\|\p_{tx}^2 \varphi\|_{\L{\infty}(\R^+, \L{1})} \Delta t \biggr) \Delta t.    
    \]

    \textbf{Estimating $D$.} Finally, we have 
    \[
        \begin{aligned}
            D
            & = \cR_{s^n}(\kappa, q^n) \varphi(y^{n+1}, t^{n+1}) \Delta t 
            + \underbrace{\frac{1}{y^{n+1} - \chi_{j_{n+1}-1}} 
            \int_{\chi_{j_{n+1}-1}^{n+1}}^{y^{n+1}} (\varphi(x, t^{n+1}) - \varphi(y^{n+1}, t^{n+1})) \Delta t}_{\eps_9} \\[5pt]
            & + \underbrace{\frac{1}{\chi_{j_{n+1}+1} - y^{n+1}} 
            \int_{y^{n+1}}^{\chi_{j_{n+1}+1}^{n+1}} (\varphi(x, t^{n+1}) - \varphi(y^{n+1}, t^{n+1})) \Delta t}_{\eps_{10}} \\[5pt]
            & = \int_{t^n}^{t^{n+1}} \cR_{s_\Delta (t)}(\kappa, q_\Delta (t)) \varphi(y_\Delta (t), t) \d{t}
            + \eps_9 + \eps_{10} + 
            \underbrace{\int_{t^n}^{t^{n+1}} \cR_{s_\Delta (t)}(\kappa, q_\Delta (t)) (\varphi(y^{n+1}, t^{n+1}) - \varphi(y_\Delta (t), t)) \d{t}}_{\eps_{11}},
        \end{aligned}
    \]

    with 
    \[
        |\eps_9| + |\eps_{10}| + |\eps_{11}| 
        \leq 2 \|f\|_{\L{\infty}} \biggl( 2\|\p_x \varphi\|_{\L{\infty}} \Delta x + \|\dot y \|_{\L{\infty}} \|\p_x \varphi\|_{\L{\infty}} \Delta t 
        + \|\p_t \varphi\|_{\L{\infty}} \Delta t \biggr) \Delta t.     
    \]
\end{proof}

Note that if $\varphi$ is supported in time in $[0, T]$, with $T \in [t^N, t^{N+1}\mathclose[$, then by summing \eqref{1Bus_aei1} over $n \in [\![0; N+1]\!]$, 
we obtain (recall that $\lambda$ is fixed):
\begin{equation}
    \label{1Bus_aei2}
    \begin{aligned}
        \int_{0}^{T} \int_{\R} & \biggl( |\rho_{\Delta} - \kappa| \p_{t} \varphi 
        + \Phi_\Delta \left( \rho_{\Delta},\kappa \right) \p_{x}\varphi \biggr) \d{x} \d{t} 
        + \int_{\R} |\rho_{\Delta}^0 -\kappa| \varphi(x, 0) \d{x} \\[5pt]
        & + \int_{0}^{T} \cR_{s_\Delta (t)}(\kappa, q_{\Delta}(t)) \varphi(y_\Delta (t), t) \d{t} \geq \bigO{\Delta x} + \bigO{\Delta t}.
    \end{aligned}
\end{equation}

We now turn to the proof of an approximate version of \eqref{1Bus_ci1}. Let us define the approximate flux function:
\[
	\bF_\Delta \left( \rho_{\Delta} \right) := \sum_{n \in \N} \left( 
        \sum_{j \leq j_n} f_{j}^n \1_{\cP_{j+1/2}^{n}} + \sum_{j \geq j_n + 1} f_{j+1}^n \1_{\cP_{j+1/2}^{n}} \right).
\]

\begin{proposition}[Approximate constraint inequalities]
    \label{1Bus_aci_pp}
    Fix $n \in \N$ and $\kappa \in [0,1]$. Then we have 
	\begin{equation}
		\label{1Bus_aci1}
		\begin{aligned}
            & \int_{y^n}^{+\infty} \rho_{\Delta}(x, t^n) \varphi(x, t^n) \d{x}
            - \int_{y^{n+1}}^{+\infty} \rho_{\Delta}(x, t^{n+1}) \varphi(x, t^{n+1}) \d{x} \\[5pt]
            & \begin{array}{cl}
                \ds{- \int_{t^n}^{t^{n+1}} \int_{\R} \biggl( \rho_{\Delta} \p_{t} \varphi 
                + \bF_\Delta \left( \rho_{\Delta} \right) \p_{x} \varphi \biggr) \d{x} \d{t}} & 
                \leq \ds{\int_{t^n}^{t^{n+1}} q_\Delta(t) \varphi(y_\Delta(t), t) \d{t}} \\[10pt]
                & + \bigO{\Delta x^2} + \bigO{\Delta x \Delta t} + \bigO{\Delta t^2}.
            \end{array}
		\end{aligned}
    \end{equation}
\end{proposition}

\begin{proof}
    Following the steps of the proof of Proposition \ref{1Bus_aei_pp}, we first multiply the scheme 
    \eqref{1Bus_mf1}-\eqref{1Bus_mf3} by $\varphi_{j+1/2}^{n+1}$, sum over $j \geq j_{n+1}$ and then apply the 
    summation by parts procedure. This time, we obtain:
    \[
        \begin{aligned}
            & \underbrace{\sum_{j \geq j_{n+1}} \rho_{j+1/2}^{n+1} \varphi_{j+1/2}^{n+1} (\chi_{j+1}^{n+1} - \chi_{j}^{n+1}) 
            - \sum_{j \geq j_n} \rho_{j+1/2}^{n} \varphi_{j+1/2}^{n} (\chi_{j+1}^{n} - \chi_{j}^{n})}_{A} \\[5pt]
            & - \underbrace{\sum_{j \geq j_n} \rho_{j+1/2}^{n} \left( \varphi_{j+1/2}^{n+1} - \varphi_{j+1/2}^{n} \right) (\chi_{j+1}^{n} - \chi_{j}^{n})}_{B} 
            + \underbrace{\sum_{j \geq j_n + 2} f_j^n \left( \varphi_{j+1/2}^{n+1} - \varphi_{j-1/2}^{n+1} \right) \Delta t}_{C}
            \leq \underbrace{q^n  \varphi_{j_{n+1} +1/2}^{n+1} \Delta t}_{D} + \eps,
        \end{aligned}
    \]

    with $\eps \leq 8 \|\p_x \varphi \|_{\L{\infty}} \Delta x^2$. Clearly,
    \[
        A = \int_{y^{n+1}}^{+\infty} \rho_{\Delta}(x, t^{n+1}) \varphi(x, t^{n+1}) \d{x} - \int_{y^{n}}^{+\infty} \rho_{\Delta}(x, t^{n}) \varphi(x, t^{n}) \d{x},
    \]

    and estimate \eqref{1Bus_aci1} follows from the bounds:
    \[
        \begin{aligned}
            & \left| B - \int_{t^n}^{t^{n+1}} \int_{\R} \rho_{\Delta} \p_{t} \varphi \d{x} \d{t} \right| \\[5pt]
            & \leq (3 \|\p_x \varphi\|_{\L{\infty}} \Delta x + \|\p_t \varphi\|_{\L{\infty}} \Delta t) \Delta t 
            + \|\dot y\|_{\L{\infty}} \biggl( 2 \|\p_x \varphi \|_{\L{\infty}} \Delta x + 2 \|\dot y\|_{\L{\infty}} \|\p_x \varphi \|_{\L{\infty}} \Delta t 
            + \|\p_t \varphi \|_{\L{\infty}} \Delta t \biggr) \Delta t \\[5pt]
            & \left| C - \int_{t^n}^{t^{n+1}} \int_{\R} \bF_\Delta \left( \rho_{\Delta} \right) \p_{x} \varphi \d{x} \d{t} \right| 
            \leq \|f\|_{\L{\infty}} \biggl( 6 \|\p_x \varphi \|_{\L{\infty}} + 4 
            \|\p_{xx}^2 \varphi \|_{\L{\infty}(\R^+, \L{1})} 
            + \|\p_{tx}^2 \varphi \|_{\L{\infty}(\R^+, \L{1})} \biggr) \Delta x \Delta t \\[5pt]
            & \left| D - \int_{t^n}^{t^{n+1}} q_\Delta(t) \varphi(y_\Delta (t), t) \d{t} \right| 
            \leq \|q\|_{\L{\infty}} \biggl( 2 \|\p_x \varphi \|_{\L{\infty}} \Delta x + \|\p_t \varphi \|_{\L{\infty}} \Delta t  
            +  \|\dot y \|_{\L{\infty}} \|\p_x \varphi \|_{\L{\infty}} \Delta t \biggr) \Delta t.
        \end{aligned}
    \]
\end{proof}

If $\varphi$ is supported in time in $(0, T)$, with $T \in [t^N, t^{N+1}\mathclose[$, then by summing \eqref{1Bus_aei1} over $n \in [\![0; N+1]\!]$, we obtain:
\begin{equation}
    \label{1Bus_aci2}
    \begin{aligned}
        & - \int_{0}^{T} \int_{\R} \biggl( \rho_{\Delta} \p_{t} \varphi 
        + \bF_\Delta \left( \rho_{\Delta} \right) \p_{x} \varphi \biggr) \d{x} \d{t}
        \leq \int_{0}^T q_\Delta(t) \varphi(y_\Delta(t), t) \d{t} + \bigO{\Delta x} + \bigO{\Delta t}.
    \end{aligned}
\end{equation}

\subsection{Compactness and convergence}
\label{1Bus_Section_compactness}

The remaining part of the reasoning consists in obtaining sufficient compactness for 
the sequence $(\rho_\Delta)_\Delta$ in order to pass to the limit in  
\eqref{1Bus_aei2}-\eqref{1Bus_aci2}. To doing so, we adapt techniques and results put 
forward by Towers in \cite{TowersOSLC}. With this in 
mind, we suppose in this section that the flux function, still bell-shaped, is such that
\begin{equation}
	\label{1Bus_uniform_concavity}
	\exists \mu > 0, \; \forall \rho \in [0,1], \quad f''(\rho) \leq - \mu.
\end{equation}

We denote for all $n \in \N$ and $j \in \Z$,
\[
	D_{j}^{n} := \max \left\{ \rho_{j-1/2}^{n} - \rho_{j+1/2}^{n}, 0 \right\}.
\]

We will also use the notation
\[
	\forall n \in \N, \quad \widehat{\Z}_{n+1} = \Z \backslash \{j_{n+1} - 2, j_{n+1} - 1, j_{n+1}, j_{n+1} + 1\}.
\]

In \cite{TowersOSLC}, the author dealt with a discontinuous in both time and space flux and the specific “vanishing viscosity” coupling at the interface. The 
discontinuity in space was localized along the curve $\{x=0\}$. Here, we deal with a smooth flux, but we have a flux constraint along the curve $\{x=y(t)\}$. The 
applicability of the technique of \cite{TowersOSLC} for our case with moving interface and flux-constrained interface coupling relies on the fact that 
one can derive a bound on $D_{j}^{n+1}$ as long as the interface does not enter the calculations for $D_{j}^{n+1}$ \textit{i.e.} as long as $j \in \widehat \Z_{n+1}$ in 
the case $j_{n+1} = j_n + 1$.

\begin{lemma}
    \label{1Bus_oslc1_lmm}
    Let $n \in \N$, $j \in \widehat{\Z}_{n+1}$, $\ds{a := \mu \frac{\Delta t}{ 4\Delta x}}$ and $\ds{\psi(x) := x - ax^2}$. Then
    \begin{equation}
        \label{1Bus_oslc1}
        D_j^{n+1} \leq \psi \left( \max\left\{D_{j-1}^n, D_j^n, D_{j+1}^n \right\} \right).
    \end{equation}
\end{lemma}

\begin{proof}
    For the sake of completeness, the proof, largely inspired by \cite{TowersOSLC}, can be found in Appendix \ref{appendix}.
\end{proof}

\begin{remark}
    Fix $n \in \N$ and $j \in \widehat{\Z}_{n+1}$. Remark that if $D_j^n > 0$, then we can write that for some $\nu(j) \in \{j-1, j, j+1\}$, we have 
    \[
        \begin{aligned}
            D_j^{n+1} 
            & \leq D_{\nu(j)}^{n} - a \left( D_{\nu(j)}^{n} \right)^2 \\[5pt] 
            & = D_{\nu(j)}^{n} \left( 1 - a D_{\nu(j)}^{n}\right) 
            = D_{\nu(j)}^{n} \frac{1 - a^2 \left( D_{\nu(j)}^{n} \right)^2}{1 + a D_{\nu(j)}^{n}}
            \leq \frac{D_{\nu(j)}^{n}}{1 + a D_{\nu(j)}^{n}} 
            = \frac{1}{\frac{1}{D_{\nu(j)}^{n}} + a}.
        \end{aligned}
    \]
\end{remark}

\begin{corollary}
    \label{1Bus_oslc2_cor}
    Let $n \in \N$. Then the scheme \eqref{1Bus_mf1}~--~\eqref{1Bus_mf3} verifies the following one-sided Lipschitz condition:
    \begin{equation}
        \label{1Bus_oslc2}
        D_j^{n+1} \leq 
        \left\{
            \begin{array}{cll}
                \ds{\frac{1}{(n+1) a}} & \text{if} & j \leq j_{n+1} - 3 - n \\[10pt]
                \ds{\frac{1}{((j_{n+1} - 2) - j) a}} & \text{if} & j_{n+1} - 3 - n \leq j \leq j_{n+1} - 3 \\[15pt]
                \ds{\frac{1}{(j - (j_{n+1} + 1)) a}} & \text{if} & j_{n+1} + 2 \leq j \leq j_{n+1} + 2 + n \\[10pt]
                \ds{\frac{1}{(n+1) a}} & \text{if} & j \geq j_{n+1} + 2 + n.
            \end{array}
        \right.
    \end{equation}
\end{corollary}

\begin{figure}[!htp]
	\begin{center}
		\includegraphics[scale = 0.20]{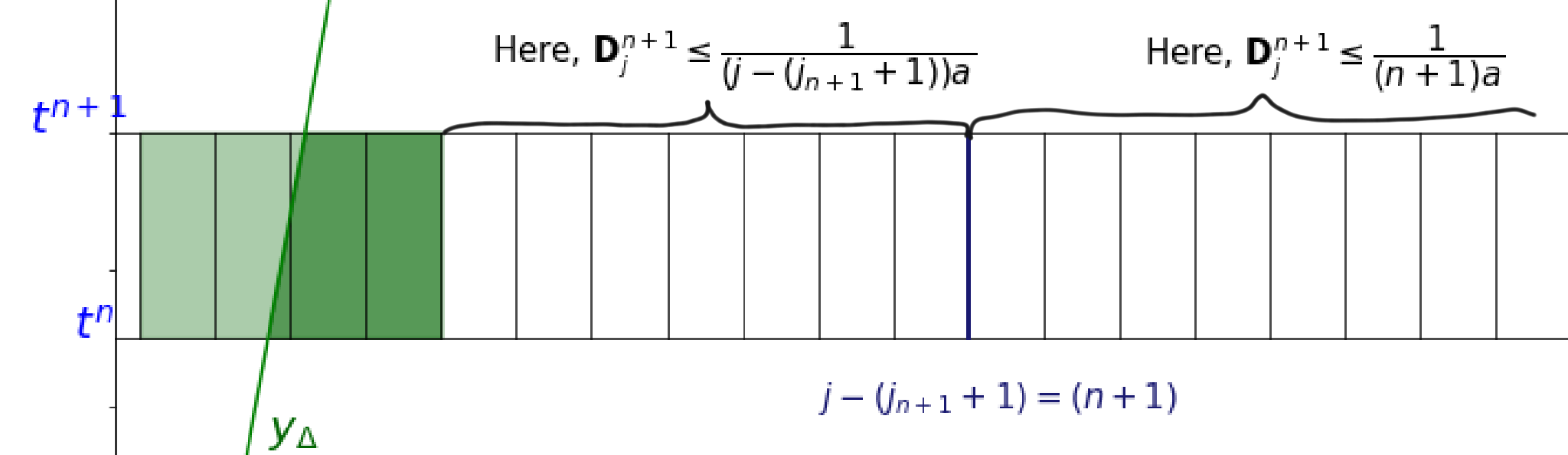}
	\end{center}
    \vspace*{-5mm}
	\caption{Illustration of the OSL bound \eqref{1Bus_oslc2}.}
	\label{fig3}
\end{figure}

\begin{proof}
    Fix $n \in \N$. We only prove \eqref{1Bus_oslc2} in the cases $j \geq j_{n+1} + 2$. The reasoning for the cases $j \leq j_0 - 3$ is very similar. 
    Let us first prove by induction on $k \in \N^*$ that
    \begin{equation}
        \label{1Bus_oslc21}
        \forall k \in \N^*, \; \forall j \in \Z, \quad \min \{n+1, j - (j_{n+1} + 1)\} \geq k \implies D_j^{n+1} \leq \frac{1}{k a}.
    \end{equation}

    Inequality \eqref{1Bus_oslc21} holds if $k=1$. Indeed, if $k=1$, then $j \geq j_{n+1} + 2$ \textit{i.e.} $j \in \widehat{\Z}_{n+1}$. By \eqref{1Bus_oslc1},
    \[
        \exists \nu_j \in \{j-1, j, j+1 \}, \quad D_j^{n+1} \leq D_{\nu_j}^{n} - a \left( D_{\nu_j}^{n} \right)^2.
    \]

    If $D_{\nu_j}^{n} = 0$, then $D_j^{n+1} = 0 \leq 1/a$. Otherwise, we can write:
    \[
        D_j^{n+1} \leq \frac{1}{\frac{1}{D_{\nu_j}^{n}} + a} \leq \frac{1}{a} = \frac{1}{k a}.
    \]

    Now, let us assume that \eqref{1Bus_oslc21} holds for some integer $k \in \N^*$ and suppose that $\min \{n+1, j - (j_{n+1} + 1) \} \geq k + 1$. Again, by 
    \eqref{1Bus_oslc1},
    \[
        \exists \nu_j \in \{j-1, j, j+1 \}, \quad D_j^{n+1} \leq D_{\nu_j}^{n} - a \left( D_{\nu_j}^{n} \right)^2.
    \]

    Since
    \[
        n \geq k \quad \text{and} \quad \nu_j - (j_{n} + 1) \geq (j - 1) - (j_{n+1} + 1) = j - (j_{n+1} + 1) - 1 \geq k,
    \]

    we deduce that $\min \{n, j - (j_{n} + 1) \} \geq k$, hence, using the induction property:
    \[
        D_j^{n+1} \leq \frac{1}{\frac{1}{D_{\nu_j}^{n}} + a} \leq \frac{1}{(k+1)a},
    \]

    which concludes the induction argument. Estimates \eqref{1Bus_oslc2} in the cases $j \geq j_{n+1} + 2$ follow for suitable choices of $k$ in 
    \eqref{1Bus_oslc21}. 
\end{proof}

\begin{corollary}[Localized $\BV$ estimates] 
    \label{1Bus_oslc3_cor}
    Fix $0 < \eps < X$ and suppose that $3 \Delta x  \leq \eps$ and that 
    $\ds{t^{n+1} \geq \frac{\eps}{2 L}}$. Then there exists a constant 
    $\ds{\Lambda = \Lambda \left( \frac{1}{\eps}, X \right)}$, nondecreasing with respect to its arguments such that
    \begin{equation}
        \label{1Bus_oslc31}
        \TV\left(\rho_\Delta(t^{n+1}) \; \1_{\mathopen]y^{n+1} + \eps, y^{n+1} + X\mathclose[} \right) \leq \Lambda,
    \end{equation}

    and 
    \begin{equation}
        \label{1Bus_oslc32}
        \int_{y^{n+1} + \eps}^{y^{n+1} + X} \left| \rho_\Delta (x, t^{n+2}) - \rho_\Delta (x, t^{n+1)}) \right| \d{x} 
        \leq 2\Delta x + L \left( 2 \Lambda + 1 \right) \Delta t.
    \end{equation}
\end{corollary}

Note that we have the same bounds for the quantities:
\[ 
    \TV\left(\rho_\Delta(t^{n+1}) \; \1_{\mathopen]y^{n+1} - X, y^{n+1} - \eps\mathclose[} \right) \; \; \text{and} \; \;
    \int_{y^{n+1} - X}^{y^{n+1} -\eps} \left| \rho_\Delta (x, t^{n+2}) - \rho_\Delta (x, t^{n+1)}) \right| \d{x}.
\] 

\begin{proof}
    Let $k_{n+1}, J_{n+1} \in \Z$ such that $y^{n+1} + \eps \in \mathopen]\chi_{k_{n+1}}^{n+1}, \chi_{k_{n+1}}^{n+1} + \Delta x\mathclose[$ and 
    $y^{n+1} + X \in \mathopen]\chi_{J_{n+1}}^{n+1}, \chi_{J_{n+1}}^{n+1} + \Delta x\mathclose[$. We have:
    \[
        \begin{aligned}
            \TV\left(\rho_\Delta(t^{n+1}) \; \1_{\mathopen]y^{n+1} + \eps, y^{n+1} + X\mathclose[} \right) 
            & = \sum_{j=k_{n+1} + 1}^{J_{n+1}} |\rho_{j+1/2}^{n+1} - \rho_{j-1/2}^{n+1}| \\[5pt]
            & = 2 \sum_{j=k_{n+1} + 1}^{J_{n+1}} D_{j}^{n+1} - \sum_{j=k_{n+1} + 1}^{J_{n+1}} (\rho_{j+1/2}^{n+1}-\rho_{j-1/2}^{n+1}) \\[5pt]
            & = 2 \sum_{j=k_{n+1} + 1}^{J_{n+1}} D_{j}^{n+1} - (\rho_{J_{n+1} - 1/2}^{n+1} - \rho_{k_{n+1} + 1/2}^{n+1}) 
            \leq 1 + 2 \sum_{j=k_{n+1} + 1}^{J_{n+1}} D_{j}^{n+1}.
        \end{aligned}
    \]

    Now, for all $j \geq k_{n+1} + 1$, we have
    \[
        \begin{aligned}
            j - (j_{n+1} + 1)
            \geq \frac{(k_{n+1} + 1) - (j_{n+1} + 1))\Delta x}{\Delta x}
            & = \frac{(\chi_{k_{n+1}}^{n+1} + \Delta x) - \chi_{j_{n+1}}^{n+1}}{^\Delta x} \\[5pt]
            & \geq \frac{(y^{n+1} + \eps) - (y^{n+1} + 2\Delta x)}{\Delta x}
            = \frac{\eps}{\Delta x} - 2 \geq 1.
        \end{aligned}
    \]
    
    Lemma \ref{1Bus_oslc2} ensures that 
    \[
        \TV\left(\rho_\Delta(t^{n+1}) \; \1_{\mathopen]y^{n+1} + \eps, y^{n+1} + X\mathclose[} \right) 
        \leq 1 + \frac{2}{a} \sum_{j=k_{n+1} + 1}^{J_{n+1}} \frac{1}{\min \{n+1, j - (j_{n+1} + 1) \}}.
    \]

    However, we also have:
    \[
        n + 1
        = \frac{t^{n+1}}{\Delta t} 
        \geq \frac{\eps}{2 L \Delta t}
        \geq \frac{\eps}{\Delta x}
        = \frac{(y^{n+1} + \eps) - y^{n+1}}{\Delta x}
        \geq \frac{\chi_{k_{n+1}}^{n+1} - (\chi_{j_{n+1}}^{n+1} + \Delta x)}{\Delta x} = k_{n+1} - (j_{n+1} + 1).
    \]

    We deduce that for all 
    $j \in [\![k_{n+1} +1; J_{n+1}]\!]$, we have $\min \{n+1, j - (j_{n+1} + 1) \} \geq k_{n+1} - (j_{n+1} + 1)$. Therefore,
    \[
        \begin{aligned}
            \sum_{j=k_{n+1} + 1}^{J_{n+1}} |\rho_{j+1/2}^{n+1} - \rho_{j-1/2}^{n+1}|
            & \leq  1 + \frac{2}{a} \times \left( \frac{J_{n+1} - k_{n+1}}{k_{n+1} - (j_{n+1} + 1)} \right) \\[5pt]
            & \leq 1 + \frac{2}{a} \times \left( \frac{X - \eps + \Delta x}{\eps - 2 \Delta x} \right) \\[5pt]
            & \leq \Lambda, \quad \Lambda:= 1 + \frac{6 X}{a \eps},
        \end{aligned}
    \]

    which is exactly \eqref{1Bus_oslc31}. Then,
    \[
        \begin{aligned}
            & \int_{y^{n+1} + \eps}^{y^{n+1} + X} \left| \rho_\Delta (x, t^{n+2}) - \rho_\Delta (x, t^{n+1)}) \right| \d{x} \\[5pt]
            & \leq 2 \Delta x + \sum_{j=k_{n+1} + 1}^{J_{n+1}} |\rho_{j+1/2}^{n+2} - \rho_{j+1/2}^{n+1}| \Delta x \\[5pt]
            & \leq 2 \Delta x + \|f'\|_{\L{\infty}} \left( \sum_{j=k_{n+1} + 1}^{J_{n+1}} |\rho_{j+3/2}^{n+1} - \rho_{j+1/2}^{n+1}|
            + \sum_{j=k_{n+1} + 1}^{J_{n+1}} |\rho_{j+1/2}^{n+1} - \rho_{j-1/2}^{n+1}| \right) \Delta t \\[5pt]
            & \leq 2\Delta x + L \left( 2 \Lambda + 1 \right) \Delta t,
        \end{aligned}
    \]

    concluding the proof.
\end{proof}

\begin{theorem}
    \label{1Bus_convergence_th}
    Fix $\rho_o \in \L{\infty}(\R, [0,1])$, $y \in \Wloc{1}{\infty}(\mathopen]0, 
    +\infty\mathclose[, \R)$ and 
    $q \in \Lloc{\infty}(\mathopen]0, +\infty\mathclose[, \R)$. 
    Suppose that $f \in \Ck{2}([0,1], \R^+)$ satisfies \eqref{bell_shaped}-\eqref
    {1Bus_uniform_concavity} and that $y$ is nondecreasing. Then as $\Delta \to 0$ 
    while satisfying the CFL condition 
    \eqref{1Bus_cfl}, $(\rho_\Delta)_\Delta$ converges a.e. on $\Omega$ to the 
    admissible entropy solution to \eqref{1Bus}.
\end{theorem}

\begin{proof}
    Fix $n \in \N^*$. The uniform convergence of $(y_\Delta)_\Delta$ to $y$, coupled with the $\BV$ bounds \eqref{1Bus_oslc31}-\eqref{1Bus_oslc32} 
    and the uniform $\L{\infty}$ bound \eqref{1Bus_scheme_stability} provide (up to a subsequence) a.e. convergence for the sequence $(\rho_\Delta)_\Delta$ 
    in any rectangular bounded domains of the open subset 
    \[ 
        O_n := \{(x, t) \in \Omega \; : \; |x - y(t)| > 1/n \},
    \]

    see \cite[Appendix A]{HRBook}. The a.e. convergence on any compact subsets of $\Omega_n$ follows by a classical covering argument. Then a diagonal procedure 
    provides the a.e. convergence on any compact subsets of 
    $O := \{(x, t) \in \Omega \; : \; x \neq y(t) \}$. 
    A further extraction yields the a.e. convergence on $\Omega$.

    Equipped with the convergence of $(\rho_\Delta)_\Delta$ to $\rho$, we let $\Delta \to 0$ in \eqref{1Bus_aei2} and \eqref{1Bus_aci2} to establish that 
    $\rho$ is an admissible entropy solution to \eqref{1Bus}. By uniqueness, the whole sequence converges to $\rho$, which proves the theorem.
\end{proof}

\begin{corollary}
    \label{1Bus_wp_cor}
    Fix $\rho_o \in \L{\infty}(\R; [0, 1])$, $y \in \Wloc{1}{\infty}(\mathopen]0, +\infty\mathclose[), \dot y \geq 0$ and $q \in \Lloc{\infty}(\mathopen]0, +\infty\mathclose[), q \geq 0$. Suppose that 
    $f \in \Ck{2}([0,1])$ satisfies \eqref{bell_shaped}-\eqref{1Bus_uniform_concavity}. Then Problem \eqref{1Bus} admits a unique admissible entropy solution.
\end{corollary}

\begin{proof}
    Existence comes from Theorem \ref{1Bus_convergence_th} while uniqueness was established by Theorem \ref{1Bus_stability_th1}.
\end{proof}

\section{Well-posedness for the multiple trajectory problem}
\label{multibus_Section_wp}

We now get back to the original problem \eqref{multibus}. Let us detail the organization of this section. First, we construct a partition of the unity to 
reduce the study of \eqref{multibus} to an assembling of several local studies of \eqref{1Bus}, see Section \ref{multibus_Section_partition_unity}. Using the 
definition based on germs, analogous to Definition \ref{1Bus_def2}, we will prove a stability estimate, leading to uniqueness, see Theorem 
\ref{multibus_stability_th}. Then in Section \ref{multibus_Section_existence}, we construct a finite volume scheme in which we fully use the precise study of 
Section \ref{1Bus_Section_existence}. A special treatment of the crossing points is described, see Section \ref{multibus_Section_mesh_scheme}.

Let us recall that we are given a finite (or more generally locally finite) family of trajectories and constraints 
$(y_i, q_i)_{i \in [\![1; J]\!]}$ defined on $\mathopen]s_i, T_i \mathclose[$ ($0 \leq s_i < T_i$). Introduce the notations:
\[
    \forall i \in [\![1;J]\!], \quad \Gamma_i := 
    \{(y_i(t), t) \; : \; t \in [s_i, T_i]\}.
\]

We suppose that for all $i \in [\![1;J]\!]$, $y_i \in \W{1}{\infty}(\mathopen]s_i, T_i \mathclose[, \R)$ and $q_i \in \L{\infty}(\mathopen]s_i, T_i \mathclose[, \R^+)$. This notation means that what 
can be seen as crossing points between interfaces will be considered as endpoints of the interfaces; for instance, given two crossing lines, we split them into 
four interfaces having a common endpoint. We denote by $(\cC_m)_{m \in [\![1;M]\!]}$ the set of all endpoints of the interfaces $\Gamma_i$, $i \in [\![1;J]\!]$.

\subsection{Reduction to a single interface}
\label{multibus_Section_partition_unity}

Fix $\varphi \in \Cc{\infty}(\overline{\Omega} \backslash \cup_{m=1}^M \cC_m, \R)$. 
Let us denote by $K$ the compact support of $\varphi$.

\textbf{Step 1.} For all $i \in [\![1;J]\!]$, $K \cap \Gamma_i$ is a compact subset (maybe empty) of $\overline{\Omega}$, and the family $(K \cap \Gamma_i)_i$ is 
pairwise disjoint. By compactness, 
\[ 
    \exists \delta > 0, \; \forall i,j \in [\![1;J]\!], \quad i \neq j \implies \text{dist}(K \cap \Gamma_i, K \cap \Gamma_j) \geq 2 \delta. 
\] 

\textbf{Step 2.} For all $i \in [\![1;J]\!]$, set 
\[ 
    \Omega_i := \bigcup_{(x, t) \in K \cap \Gamma_i} \textbf{B}((x, t), \delta),
\] 

where $\textbf{B}((x, t), \delta)$ denotes the $\R^2$-euclidean open ball centered on 
$(x, t)$ and of radius $\delta$. Clearly, $\Omega_i$ is an open subset 
of $\overline{\Omega}$ containing $\Gamma_i$. Moreover, the family $(\Omega_i)_i$ is pairwise disjoint. Indeed, suppose instead that for some 
$i, j \in [\![1;J]\!]$ ($i \neq j$), we have $\Omega_i \cap \Omega_j \neq \emptyset$,
and fix $(x, t) \in \Omega_i \cap \Omega_j$. By definition, there exists $(x_i, t_i) \in K \cap \Gamma_i$ and $(x_j, t_j) \in K \cap \Gamma_j$ such that 
\[
    (x, t) \in \textbf{B}((x_i, t_i), \delta) \cap \textbf{B}((x_j, t_j), \delta).
\]

Using the triangle inequality, we deduce that 
\[
    \text{dist}(K \cap \Gamma_i, K \cap \Gamma_j)
    \leq \text{dist}((x_i, t_i), (x_j, t_j))
    \leq \text{dist}((x_i, t_i), (x, t)) + \text{dist}((x, t), (x_j, t_j))
    < 2 \delta,
\]

yielding the contradiction.

\textbf{Step 3.} Define the open subset (finite intersection of open subsets):
\[
    \Omega_o := 
    \left\{(x, t) \in \overline{\Omega} \; : \; \forall i \in [\![1;J]\!], \; \text{dist}((x, t), K \cap \Gamma_i) \geq \frac{\delta}{2} \right\}.
\]

The family $(\Omega_i)_{i \in [\![0;J]\!]}$ is an open covering of $\R \times \R^+$. Consequently, there exists a partition of the unity 
$(\theta_i)_{i \in [\![0;J]\!]}$ associated with this covering:
\[
    \forall i \in [\![0;J]\!], \; \theta_i \in \Cc{\infty}(\Omega_i, \R^+); 
    \quad \forall (x,t) \in \R \times \R^+, \; \sum_{i=0}^J \theta_i(x, t) = 1.   
\] 

\textbf{Step 4.} We write the function $\varphi$ in the following manner:
\begin{equation}
    \label{partition_unity}
    \varphi = \sum_{i=0}^J (\varphi \theta_i) = \varphi_o + \sum_{i=1}^J \varphi_i.
\end{equation}

Note that:
\begin{enumerate}
    \item $\varphi_o$ vanishes along all the interfaces;
    \item for all $i \in [\![1;J]\!]$, $\varphi_i$ vanishes along all the interfaces but $\Gamma_i$.
\end{enumerate}

\subsection{Definition of solutions and uniqueness}

Following Section \ref{1Bus_Section_uniqueness} and Definition \ref{1Bus_def2}, we give the following definition of solution.

\begin{definition}
    \label{multibus_def1}
    Let $\rho_o \in \L{\infty}(\R, [0, 1])$. We say that $\rho \in \L{\infty}(\Omega, [0, 1])$ is a $\cG$-entropy solution to \eqref{multibus} if:

    (i) for all test functions $\varphi \in \Cc{\infty}(\overline{\Omega} \backslash \cup_{i=1}^J \Gamma_i, \R^+)$ and $\kappa \in [0,1]$, the 
    following entropy inequalities are verified:
	\begin{equation}
		\label{multibus_ei1}
        \int_{0}^{+\infty} \int_{\R} \biggl( |\rho-\kappa| \p_{t}\varphi + \Phi(\rho,\kappa) \p_{x} \varphi \biggr) \d x \d t 
        + \int_{\R}|\rho_o(x)-\kappa|\varphi(x,0) \d x \geq 0;
    \end{equation}

    (ii) for all $i \in [\![1;J]\!]$ and for a.e. $t \in \mathopen]s_i, T_i \mathclose[$,
	\begin{equation}
        \label{multibus_ci1}
        (\rho(y_i(t)-,t), \rho(y_i(t)+,t)) \in \cG_{\dot y_i(t)}(q_i(t)),
    \end{equation}

    where the admissibility germ $\cG_{\dot y_i}(q_i)$ is defined in Definition \ref{1Bus_def_germ}.
\end{definition}

\begin{lemma}[Kato inequality]
    \label{multibus_Kato_lmm}
    Fix $\rho_o, \sigma_o \in \L{\infty}(\R, [0,1])$. Let $(q_i)_{i \in [\![1;J]\!]}$ and $(\overset{\sim}{q}_i)_{i \in [\![1;J]\!]}$ be two family of constraints, 
    where for all $i \in [\![1;J]\!]$, $q_i$, $\overset{\sim}{q}_i \in \L{\infty}(\mathopen]s_i, T_i \mathclose[, \R)$. We denote by $\rho$ (resp. $\sigma$) a $\cG$-entropy solution to 
    Problem \eqref{multibus} corresponding to initial datum $\rho_o$ (resp. $\sigma_o$) and constraints $(q_i)_{i \in [\![1;J]\!]}$ 
    (resp. $(\overset{\sim}{q}_i)_{i \in [\![1;J]\!]}$). Then for all test functions $\varphi \in \Cc{\infty}(\overline{\Omega}, \R^+)$, we have
    \begin{equation}
        \label{multibus_Kato}
        \begin{aligned}
            & \int_{0}^{+\infty} \int_{\R} \biggl( |\rho- \sigma| \p_{t}\varphi + \Phi(\rho, \sigma) \p_{x} \varphi \biggr) \d x \d t  
            + \int_{\R} |\rho_o(x) - \sigma_o(x)| \varphi(x,0) \d x \\[5pt]
            & + \sum_{i=1}^J \int_{s_i}^{T_i} \biggl( \Phi_{\dot y_i(t)} \left(\rho(y_i(t)+,t), \sigma(y_i(t)+,t) \right) 
            - \Phi_{\dot y_i(t)} \left(\rho(y_i(t)-,t), \sigma(y_i(t)-,t) \right) \biggr) \varphi(y_i(t),t) \d{t} \geq 0.
        \end{aligned}
	\end{equation}
\end{lemma}

\begin{proof}
    We split the reasoning in two steps.

    \textbf{Step 1.} Suppose first that $\varphi \in \Cc{\infty}(\overline{\Omega} \backslash \cup_{m=1}^M \cC_m, \R^+)$. In this case, we write $\varphi$ using the 
    partition of unity \eqref{partition_unity}. Fix $i \in [\![1;J]\!]$. Following the computations of Lemma \ref{1Bus_Kato1_lmm}, we obtain:
    \begin{equation}
        \label{multibus_Kato_eq2}
        \begin{aligned}
            & \iint_{\Omega_i} \biggl( |\rho- \sigma| \p_{t}\varphi_i + \Phi(\rho, \sigma) \p_{x} \varphi_i \biggr) \d x \d t  
            + \int_{\{x \in \R \; : \; (x, 0) \in \Omega_i\}} |\rho_o(x) - \sigma_o(x)| \varphi_i(x,0) \d x \\[5pt]
            & + \int_{s_i}^{T_i} \biggl( \Phi_{\dot y_i(t)} \left(\rho(y_i(t)+,t), \sigma(y_i(t)+,t) \right) 
            - \Phi_{\dot y_i(t)} \left(\rho(y_i(t)-,t), \sigma(y_i(t)-,t) \right) \biggr) \varphi_i(y_i(t),t) \d{t} \geq 0.
        \end{aligned}
    \end{equation}

    Now, since $\varphi_o$ vanishes along all the interfaces, standard computations lead to
    \begin{equation}
        \label{multibus_Kato_eq3}
        \iint_{\Omega_o} \biggl( |\rho- \sigma| \p_{t}\varphi_o + \Phi(\rho, \sigma) \p_{x} \varphi_o \biggr) \d x \d t  
        + \int_{\{x \in \R \; : \; (x, 0) \in \Omega_o\}} |\rho_o(x) - \sigma_o(x)| \varphi_o(x,0) \d x \geq 0.
    \end{equation}

    We now sum \eqref{multibus_Kato_eq2} ($i \in [\![1;J]\!]$) and \eqref{multibus_Kato_eq3} to obtain \eqref{multibus_Kato}. This inequality is 
    analogous to \eqref{1Bus_Kato1}.

    \textbf{Step 2.} Consider now $\varphi \in \Cc{\infty}(\overline{\Omega}, \R^+)$. 
    Fix $n \in \N^*$. From the first step, a classical approximation 
    argument allows us to apply \eqref{multibus_Kato} with the Lipschitz test function
    \[ 
        \psi_n (x, t) = \left( \sum_{m=1}^M \delta_{m, n}(x, t) \right) \varphi(x, t),    
    \]

    where for all $m \in [\![1;M]\!]$,
    \[
        \delta_{m, n}(x, t) = 
        \left\{
            \begin{array}{ccc}
                0 & \text{if} & \ds{\text{dist}_1((x, t), \cC_m) < \frac{1}{n}} \\[5pt]
                \ds{n \left(\text{dist}_1((x, t), \cC_m) - \frac{1}{n} \right)} 
                & \text{if} & \ds{\frac{1}{n} \leq \text{dist}_1((x, t), \cC_m) \leq \frac{2}{n}} \\[5pt]
                1 & \text{if} & \ds{\text{dist}_1((x, t), \cC_m) > \frac{2}{n}},
            \end{array}
        \right.
    \] 

    where, by analogy with the proof of Lemma \ref{1Bus_Kato1_lmm}, $\text{dist}_1$ denotes the $\R^2$ distance associated with the norm $\|\cdot\|_1$. We 
    let $n \to +\infty$, keeping in mind that:
    \[ 
        \left\| \left(\sum_{m=1}^M \delta_{m, n}\right) \varphi - \varphi \right\|_{\L{1}(\Omega, \R)} \limit{n}{+\infty} 0; \quad 
        \forall m \in [\![1 ;M]\!], \; \left\| \nabla \delta_{m, n} \right\|_{\L{1}(\Omega, \R^2)} = \bigO{\frac{1}{n}}.
    \]

    Straightforward computations lead to \eqref{multibus_Kato} with $\varphi \in \Cc{\infty}(\overline{\Omega}, \R)$, concluding the proof.
\end{proof}

\begin{theorem}
    \label{multibus_stability_th}
    Fix $\rho_o, \sigma_o \in \L{\infty}(\R, [0,1])$. Let 
    $(q_i)_{i \in [\![1;J]\!]}$ and $(\overset{\sim}{q}_i)_{i \in [\![1;J]\!]}$ be two 
    family of constraints, where for all $i \in [\![1;J]\!]$, $q_i$, 
    $\overset{\sim}{q}_i \in \L{\infty}(\mathopen]s_i, T_i \mathclose[, \R)$.
     We denote by $\rho$ (resp. $\sigma$) a $\cG$-entropy solution to 
    Problem \eqref{multibus} corresponding to initial datum $\rho_o$ (resp. $\sigma_o$) 
    and constraints $(q_i)_{i \in [\![1;J]\!]}$ 
    (resp. $(\overset{\sim}{q}_i)_{i \in [\![1;J]\!]}$). Then for all $T > 0$, we have 
    \begin{equation}
        \label{multibus_stability}
        \|\rho(T) - \sigma(T)\|_{\L{1}(\R)} \leq \|\rho_o - \sigma_o\|_{\L{1}(\R)} 
        + \sum_{i =1}^J 2 \int_{s_i}^{T_i} \left| q_i(t) - \overset{\sim}{q}_i(t) \right| \d{t}.
    \end{equation}

    In particular, Problem \eqref{multibus} admits at most one $\cG$-entropy solution.
\end{theorem}

\begin{proof}
    Estimate \eqref{multibus_stability} follows from Kato inequality \eqref{multibus_Kato} with a suitable choice of test function and in light of the inequality:
    \[
        \begin{aligned}
            & \forall i \in [\![1;J]\!], \; \text{for a.e.} \; t \in \mathopen]s_i, T_i \mathclose[, \\[5pt]
            & \Phi_{\dot y_i(t)} \left(\rho(y_i(t)+,t), \sigma(y_i(t)+,t) \right) - \Phi_{\dot y_i(t)} \left(\rho(y_i(t)-,t), \sigma(y_i(t)-,t) \right)
            \leq 2 |q_i(t) - \overset{\sim}{q}_i(t)|,
        \end{aligned}
    \]

    see Theorem \ref{1Bus_stability_th1} and its proof.
\end{proof}

\subsection{Proof of existence}
\label{multibus_Section_existence}

Following the reasoning of Sections \ref{1Bus_Section_uniqueness}-\ref{1Bus_Section_existence}, we introduce a second definition of solutions, more suitable to 
prove existence.

\begin{definition}
    \label{multibus_def2}
    Let $\rho_o \in \L{\infty}(\R, [0, 1])$. We say that 
    $\rho \in \L{\infty}(\Omega, [0, 1])$ is an admissible entropy solution to \eqref{multibus} if

    (i) for all test functions $\varphi \in \Cc{\infty}(\overline{\Omega}, \R^+)$ and $\kappa \in [0,1]$, the following entropy inequalities are verified:
	\begin{equation}
        \label{multibus_ei2}
        \begin{aligned}
            \int_{0}^{+\infty} \int_{\R} & \biggl( |\rho-\kappa| \p_{t}\varphi + \Phi(\rho,\kappa) \p_{x} \varphi \biggr) \d x \d t  
            + \int_{\R} |\rho_o(x)-\kappa|\varphi(x,0) \d x \\[5pt]
            & + \sum_{i=1}^J \int_{s_i}^{T_i} \cR_{\dot y_i(t)}(\kappa, q_i(t)) \varphi(y_i(t),t) \d t \geq 0,
        \end{aligned}
	\end{equation}

    where $\cR_{\dot y_i}(\kappa, q_i)$ is defined in Definition \ref{1Bus_def1} ;

    (ii) for all test functions $\varphi \in \Cc{\infty}(\Omega \backslash \cup_{m=1}^M \cC_m, \R^+)$, written under the form 
    \eqref{partition_unity}, the following constraint inequalities are verified for all $i \in [\![1;J]\!]$:
	\begin{equation}
        \label{multibus_ci2}
        - \iint_{\Omega_i^+} \biggl( \rho \p_{t} \varphi + f(\rho) \p_x \varphi \biggr) \d x \d t \leq \int_{s_i}^{T_i} q_i(t) \varphi_i(y_i(t), t) \d t,
    \end{equation}

    where $\ds{\Omega_i^+ := \left\{ (x,t) \in \Omega_i \; : \; x > y_i(t) \right\}}$.
\end{definition}

\begin{proposition}
    \label{multibus_defs_equivalent}
    Definition \ref{multibus_def1} and Definition \ref{multibus_def2} are equivalent. Moreover, in Definition \ref{multibus_def2} (i), it is equivalent that 
    \eqref{multibus_ei2} holds with $\varphi \in \Cc{\infty}(\overline{\Omega} \backslash \cup_{m=1}^M \cC_m, \R^+)$.
\end{proposition}

\begin{proof}
    The proof of the equivalence of Definitions \ref{multibus_def1} and \ref{multibus_def2} is a straightforward adaptation of the proofs of 
    Propositions \ref{1Bus_def12_pp}-\ref{1Bus_def21_pp}. The last part of the statement follows using the same approximation argument described at the end of 
    the proof of Lemma \ref{multibus_Kato_lmm}.
\end{proof}

Let us now turn to the proof of existence for admissible entropy solutions of \eqref{multibus}. We make use of the precise study of Section 
\ref{1Bus_Section_existence} in the case of a single trajectory and build a finite volume scheme. We keep the notations of Section \ref{1Bus_Section_existence} 
when there is no ambiguity.

\subsubsection{Construction of the mesh, definition of the scheme}
\label{multibus_Section_mesh_scheme}

For the sake of clarity, suppose that we only have two trajectories/constraints $(y_i, q_i)$, $i \in \{1, 2\}$, defined on $[0, \tau]$, which cross at time 
$\tau$. We denote by $\cC$ this crossing point. Suppose also that this crossing point results in two additional trajectories/constraints $(y_i, q_i)$, 
$i \in \{3, 4\}$, defined on $[\tau, T]$, and which do not cross, as represented in Figure \ref{fig4}.

Let us fully make explicit the steps of the reasoning leading to the construction of our scheme in that situation. Suppose that $\lambda = \Delta t / \Delta x$ is 
fixed and verifies the CFL condition
\begin{equation}
	\label{multibus_cfl}
    2 \left(\underbrace{\|f'\|_{\L{\infty}} + \max_{i \in [\![1; 4]\!]} \|\dot y_i \|_{\L{\infty}(\mathopen]0, T\mathclose[)}}_{:=L} \right) \lambda \leq 1.
\end{equation}

Set $N \in \N$ such that $\tau \in [t^N, t^{N+1}\mathclose[$. We divide the discussion in four parts.

\textbf{Part 1.} Introduce the number
\[ 
    N_1 := \inf \left \{n \in \N \; : \; |y_\Delta^1 (t^n) - y_\Delta^2(t^n)| \leq 4 \Delta x \right \}.    
\]

The definition of $N_1$ ensures that for all $n \in [\![0; N_1 -1]\!]$, we can independently modify the mesh near the two trajectories $y_\Delta^1$ and 
$y_\Delta^2$, as presented in Figure \ref{fig5}. 

\newpage

Consequently, we can simply define the approximate solution $\rho_\Delta$ on $\R \times [0, t^{N_1 - 1}]$ as 
the finite volume approximation of a conservation law, with initial datum $\rho_o$, with flux constraints on two non-interacting trajectories, using the recipe of 
Section \ref{1Bus_Section_existence} for each trajectory/constraint.

\begin{figure}[!htp]
	\begin{center}
		\includegraphics[scale = 0.7]{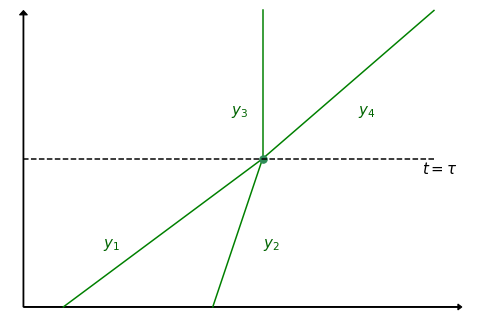}
	\end{center}
    \vspace*{-5mm}
	\caption{Illustration of the configuration.}
	\label{fig4}
\end{figure}

\textbf{Part 2.} Fix now $n \in [\![N_1; N]\!]$. In these time intervals, since the two trajectories are too close to each other, one cannot modify the 
mesh in the neighborhood of one of them without affecting the other. However, the scheme has to be defined globally, so we proceed as described below.
\begin{itemize}
    \item First, introduce the mean trajectory and the new constraint:
    \[
        \forall t \in [0, \tau], \quad y_{12}(t) := \frac{y_1(t) + y_2(t)}{2}; \quad q_{12}(t) := \min\{q_1(t), q_2(t)\},
    \]

    represented in purple in Figure \ref{fig5}, before the crossing point (in red). The choice of taking the minimal level of constraint in the definition 
    of $q_{12}$ stems from the nature of the constrained problem; see however Remark \ref{crossing_points_rk} below.
    \item Then, define $\rho_\Delta$ on $\R \times [t^{N_1}, t^{N}]$ as the finite volume approximation of the one trajectory/one constraint problem:
    \[
        \left\{
            \begin{aligned}
                \p_{t}\rho + \p_{x} \left( f(\rho) \right) & = 0 \\
                \rho(\cdot, t^{N_1}) & = \rho_{\Delta}(\cdot, t^{N_1 - 1}) \\
                \left. \left(f(\rho) - \dot y_{12}(t) \rho \right) \right|_{x= y_{12}(t)} & \leq q_{12}(t) & t \in \mathopen]t^{N_1}, t^{N}\mathclose[,
            \end{aligned}
        \right.
    \]

    using exactly the recipe of Section \ref{1Bus_Section_mesh_scheme}.
\end{itemize}

\textbf{Part 3.} Introduce the number:
\[ 
    N_2 := \inf \left \{n \in \N \; : \; n > N \;\; \text{and} \;\; 
    |y_\Delta^3 (t^n) - y_\Delta^4(t^n)| \geq 4 \Delta x \right \}.    
\]

For $n \in [\![N; N_2]\!]$, we are in the same situation as Part 2. We proceed to the same construction, \textit{mutatis mutandis.}
\begin{itemize}
    \item As in Part 2, define the mean trajectory and the new constraint:
    \[
        \forall t \in [\tau, T], \quad y_{34}(t) := \frac{y_3(t) + y_4(t)}{2}; \quad q_{34}(t) := \min\{q_3(t), q_4(t)\},
    \]

    represented in purple in Figure \ref{fig5}, after the crossing point.
    \item Define $\rho_\Delta$ on $\R \times [t^{N}, t^{N_2}]$ as the finite volume approximation of the one trajectory/one constraint problem:
    \[
        \left\{
            \begin{aligned}
                \p_{t}\rho + \p_{x} \left( f(\rho) \right) & = 0 \\
                \rho(\cdot, t^N) & = \rho_{\Delta}(\cdot, t^{N}) \\
                \left. \left(f(\rho) - \dot y_{34}(t) \rho \right) \right|_{x= y_{34}(t)} & \leq q_{34}(t) & t \in \mathopen]t^{N}, t^{N_2}\mathclose[.
            \end{aligned}
        \right.
    \]
\end{itemize}

\begin{figure}[!htp]
	\begin{center}
		\includegraphics[scale = 0.7]{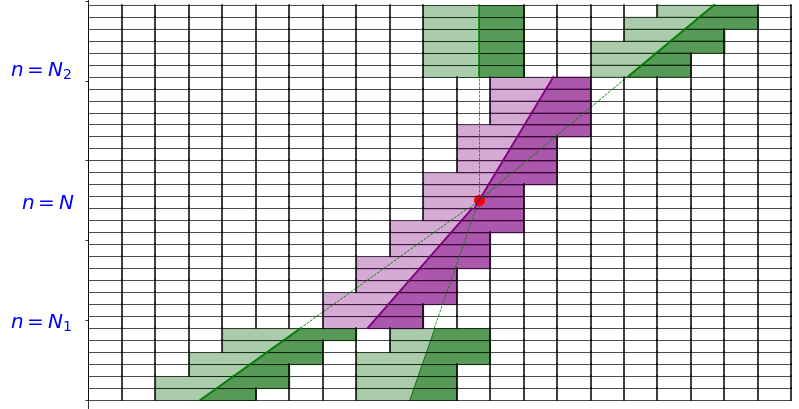}
	\end{center}
    \vspace*{-5mm}
	\caption{Illustration of the local modifications of the mesh.}
	\label{fig5}
\end{figure}

\textbf{Part 4.} Finally, $\rho_\Delta$ is defined on $\R \times [t^{N_2}, T]$ like in Part 1 with $y_3, q_3, \rho_{\Delta}(\cdot, t^{N_2})$ (resp. $y_4, q_4$) 
playing the role of $y_1, q_1, \rho_o$ (resp. of $y_2, q_2$).

\begin{remark}
    \label{crossing_points_rk}
    Let us stress out that the details of the treatment done in Parts 2-3 do not play any significant role in the convergence proof below thanks to the choice 
    of test functions vanishing at neighborhood of the crossing points, see Proposition \ref{multibus_defs_equivalent}. Consequently, taking the mean trajectory 
    and the minimum of the constraint is merely an example aiming at preserving some consistency while keeping the scheme simple to understand and implement.
\end{remark}  

The general case of a finite number of interfaces (locally finite number can be easily included) is treated in the same way, leading to a pattern with the 
uniform rectangular mesh adapted to each of the interfaces $\Gamma_i$, $i \in [\![1;J]\!]$ except for small (in terms of the number of impacted mesh cells) 
neighborhoods of the crossing points $\cC_m$, $m \in [\![1;M]\!]$.

\subsubsection{Proof of convergence}

\begin{theorem}
    \label{multibus_convergence_th}
    Fix $T> 0$, $f \in \Ck{2}([0,1], \R^+)$ satisfying \eqref{bell_shaped}-\eqref{1Bus_uniform_concavity} and $\rho_o \in \L{\infty}(\R, [0,1])$. Let 
    $(y_i, q_i)_{i \in [\![1; J]\!]}$ be a finite family of trajectories and constraints defined on $\mathopen]s_i, T_i \mathclose[$ ($0 \leq s_i < T_i$). We suppose that for all 
    $i \in [\![1; J]\!]$, $y_i \in \W{1}{\infty}(\mathopen]s_i, T_i \mathclose[, \R)$ and $q_i \in \L{\infty}(\mathopen]s_i, T_i \mathclose[, \R^+)$. Suppose also that the interfaces $(\Gamma_i)_i$ 
    defined by the trajectories $(y_i)_i$ have a finite number of crossing points. Then as $\Delta \to 0$ while satisfying the CFL condition
    \[
        2 \left(\underbrace{\|f'\|_{\L{\infty}} + \max_{i \in [\![0; J]\!]} \|\dot y_i \|_{\L{\infty}(\mathopen]0, T\mathclose[)}}_{:=L} \right) \lambda \leq 1,
    \]
    
    the sequence $(\rho_\Delta)_\Delta$ constructed by the procedure of Section \ref{multibus_Section_mesh_scheme} converges a.e. on $\Omega$ to the admissible 
    entropy solution to \eqref{multibus}.
\end{theorem}

\begin{proof}
    We make use of the fact that in Definition \ref{multibus_def2}, we only need to consider test functions that vanish at a neighborhood of the crossing 
    points (this is the key observation leading to Remark \ref{crossing_points_rk} here above).
    
    \textit{(i) Proof of the entropy inequalities.} Fix $\varphi \in \Cc{\infty}(\overline{\Omega} \backslash \cup_{m=1}^M \cC_m, \R^+)$, written 
    as $\varphi = \varphi_o + \sum_{i=1}^J \varphi_i$, using the appropriate partition of unity, see Section \ref{multibus_Section_partition_unity}. Since 
    $\varphi_o$ vanishes along all the interfaces, $\rho_\Delta$ verifies inequality \eqref{1Bus_aei2} with $\cR \equiv 0$ on the domain $\Omega_o$ and 
    with test function $\varphi_o$. Indeed, for a sufficiently small $\Delta x > 0$, the scheme we constructed in the previous section reduces to a standard 
    finite volume in $\Omega_o$. Fix now $i \in [\![1;J]\!]$. Since $\varphi_i$ vanishes along all the interfaces but $\Gamma_i$, $\rho_\Delta$ verifies 
    inequality \eqref{1Bus_aei2} with reminder term $\cR_{s_\Delta^{i}}(\kappa, q_\Delta^{i})$ along the trajectory $y_\Delta^{i}$ on the domain $\Omega_i$ and 
    with test function $\varphi_i$, due to the analysis of Section \ref{1Bus_Section_existence}; indeed, in the support of the test function, our scheme for the 
    multi-interface problem reduces to the scheme for the single-interface problem. By summing these previous inequalities, we obtain an approximate version of 
    \eqref{multibus_ei2} verified by $\rho_\Delta$:
    \begin{equation}
        \label{multibus_aei}
        \begin{aligned}
            \int_{0}^{+\infty} \int_{\R} & \biggl( |\rho_\Delta - \kappa| \p_{t}\varphi + \Phi_\Delta (\rho_\Delta ,\kappa) \p_{x} \varphi \biggr) \d x \d t  
            + \int_{\R} |\rho_\Delta^0(x)-\kappa| \varphi(x, 0) \d x \\[5pt]
            & + \sum_{i=1}^J \int_{s_i}^{T_i} \cR_{s_\Delta^{i}(t)}(\kappa, q_\Delta^{i}(t)) \varphi(y_\Delta^{i}(t), t) \d t \geq 
            \bigO{\Delta x} + \bigO{\Delta t}.
        \end{aligned}
	\end{equation}

    \textit{(ii) Proof of the weak constraint inequalities.} Let $\varphi \in \Cc{\infty}(\Omega \backslash \cup_{m=1}^M \cC_m, \R^+)$, written under 
    the form \eqref{partition_unity}. Fix $i \in [\![1;J]\!]$. Since $\varphi_i$ vanishes along all the interfaces but $\Gamma_i$, for a sufficiently small 
    $\Delta x$, $\rho_\Delta$ verifies inequality \eqref{1Bus_aci2} with constraint $q_\Delta^{i}$ along the trajectory $y_\Delta^{i}$ on the domain $\Omega_i^+$ 
    and with test function $\varphi_i$. We obtain an approximate version of \eqref{multibus_aci} verified by $\rho_\Delta$:
    \begin{equation}
        \label{multibus_aci}
        - \iint_{\Omega_i^+} \biggl( \rho_\Delta \p_{t} \varphi + \bF_\Delta (\rho_\Delta) \p_x \varphi \biggr) \d x \d t 
        \leq \int_{s_i}^{T_i} q_\Delta^{i}(t) \varphi_i(y_\Delta^{i}(t), t) \d t + \bigO{\Delta x} + \bigO{\Delta t}.
    \end{equation}

    \textit{(iii) Compactness and convergence.} Compactness of the sequence $(\rho_\Delta)_\Delta$ follows directly from the study of Section 
    \ref{1Bus_Section_compactness} where we derived local $\BV$ bounds for $(\rho_\Delta)_\Delta$ under the assumption \eqref{1Bus_uniform_concavity}.
    Indeed, these local bounds lead to compactness in the domain complementary to the interfaces, we only use the fact that the interfaces together with the 
    crossing points form a closed subset of $\Omega$ with zero Lebesgue measure. Once the a.e. convergence (up to a subsequence) on $\Omega$ to some 
    $\rho \in \L{\infty}(\Omega, [0, 1])$ obtained, we simply pass to the limit in \eqref{multibus_aei}-\eqref{multibus_aci}. This proves that $\rho$ is an 
    admissible solution to \eqref{multibus}. By the uniqueness of Theorem \ref{multibus_stability_th}, the whole sequence converges to $\rho$. This concludes the 
    proof.
\end{proof}

\begin{corollary}
    \label{multibus_wp_cor}
    Fix $T> 0$, $f \in \Ck{2}([0,1], \R^+)$ satisfying \eqref{bell_shaped}-\eqref{1Bus_uniform_concavity} and $\rho_o \in \L{\infty}(\R, [0,1])$. Let 
    $(y_i, q_i)_{i \in [\![1; J]\!]}$ be a finite family of trajectories and constraints defined on $\mathopen]s_i, T_i \mathclose[$ ($0 \leq s_i < T_i$). We suppose that for all 
    $i \in [\![1; J]\!]$, $y_i \in \W{1}{\infty}(\mathopen]s_i, T_i \mathclose[, \R)$ and $q_i \in \L{\infty}(\mathopen]s_i, T_i \mathclose[, \R^+)$. Suppose also that the interfaces $(\Gamma_i)_i$ 
    defined by the trajectories $(y_i)_i$ have a finite number of crossing points.
    Then Problem \eqref{multibus} admits a unique admissible entropy solution.
\end{corollary}

\begin{proof}
    Existence comes from Theorem \ref{multibus_convergence_th} while uniqueness is established by Theorem \ref{multibus_stability_th}.
\end{proof}

\section{Numerical experiment with crossing trajectories}

In this section, we perform a numerical test to illustrate the scheme analyzed in 
Section \ref{1Bus_Section_existence} and Section \ref{multibus_Section_existence}. We 
take the GNL flux $f(\rho) = \rho(1 - \rho)$.

We model the following situation. A vehicle breaks down on a road and reduces by half the surrounding traffic flow, which initial state is given by 
$\ds{\rho_o = 0.8 \times \1_{[1, 3]}}$. At some point, a tow truck comes to move the immobile vehicle. We summarized this situation in Figure \ref{fig6}. 
Notice the time interval in which $q_3 \equiv 0.1$. This corresponds to the time needed for the tow truck to move the vehicle. Remark also that the value of the 
constraint on this time interval is smaller than the one when only the broken down vehicle was reducing the traffic flow.

\begin{figure}[!htp]
	\begin{center}
		\includegraphics[scale = 0.66]{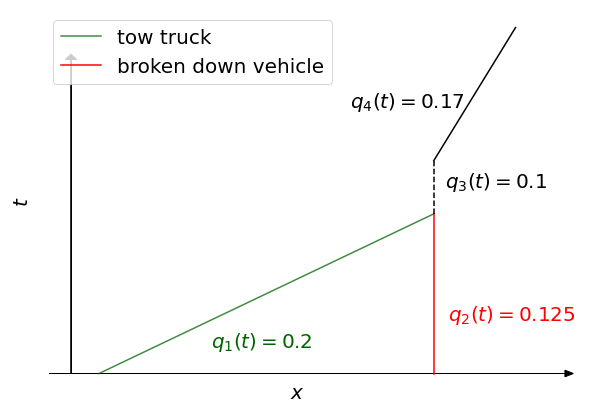}
	\end{center}
	\vspace*{-5mm}
	\caption{A tow truck comes moving an immobile vehicle.}
	\label{fig6}
\end{figure}

The evolution of the numerical solution is represented in Figure \ref{fig7}. Let us comment on the profile of the numerical solution.
\begin{itemize}
    \item At first ($0 \leq t \leq 5.80$), the solution is composed of traveling waves separated by a stationary non-classical shock located at the immobile 
    vehicle position.
    \item When the tow truck catches up with the vehicle ($6.30 \leq t \leq 8.0$), the profile of the numerical solution is the same, but the greater value of the 
    constraint in this time interval changes the magnitude of the non-classical shock; at this point the combined presence of both the tow truck and the 
    immobile vehicle clogs the traffic flow even more. 
    \item Finally, once the tow truck starts again ($t > 8.0$), the traffic congestion is reduced.
\end{itemize}

Notice at time $t=7.44$ the small artifact (circled in red in Figure \ref{fig7}) created by Parts 2-3 in the construction of the approximate solution and 
reproduced by the scheme. This highlights the fact that even if the treatment of the crossing points brings inconsistencies or artifacts to the numerical solution, 
these undesired effects are not amplified by the scheme, and become negligible when one refines the mesh.

\begin{figure}[!htp]
	\begin{center}
		\includegraphics[scale = 0.75]{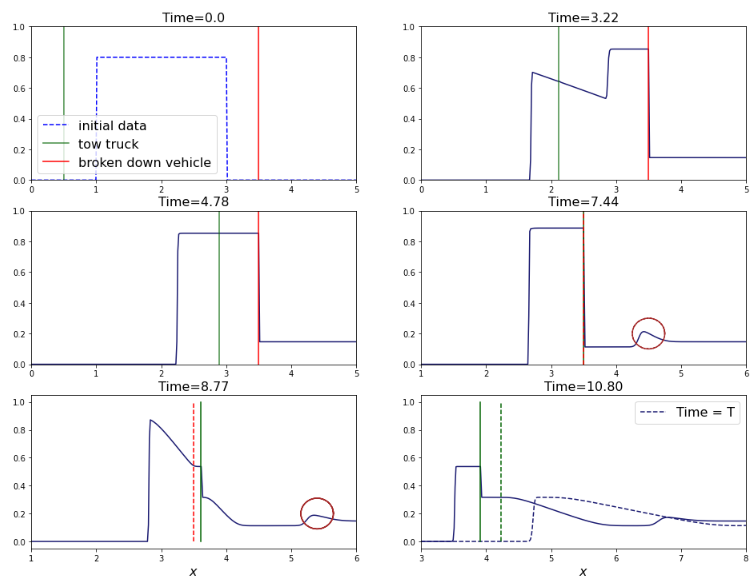}
	\end{center}
	\vspace*{-5mm}
	\caption{The numerical solution at different fixed times; for an animated evolution of the solution, follow: 
    \href{https://www.abrahamsylla.com/numerical-simulations}{https://www.abrahamsylla.com/numerical-simulations}}
	\label{fig7}
\end{figure}

\newpage 

\appendix 

\section{Proof of the OSL bound}
\label{appendix}

We prove in this appendix Lemma \ref{1Bus_oslc1_lmm}. All the notations are taken from Sections \ref{1Bus_Section_mesh_scheme} and 
\ref{1Bus_Section_compactness}. The proof is a simple rewriting of the proof of \cite[Lemma 4.2]{TowersOSLC}.

It will be convenient to write the Engquist-Osher flux under the form:
\[
    \forall a,b \in [0,1], \quad 
    \eo(a, b) = \underbrace{\left( f(a \wedge \overline{\rho}) - \frac{f(\overline{\rho})}{2} \right)}_{q_+(a)} + 
    \underbrace{\left( f(b \vee \overline{\rho}) - \frac{f(\overline{\rho})}{2} \right)}_{q_-(b)},
\]

so that for all $n \in \N$, when $j \in \widehat{\Z}_{n+1}$, the scheme \eqref{1Bus_mf1} can be rewritten as:
\begin{equation}
    \label{mf_eo}
    \rho_{j+1/2}^{n+1} = 
    \rho_{j+1/2}^{n} - \lambda \biggl( q_+\left( \rho_{j+1/2}^{n}\right) + q_-\left( \rho_{j+3/2}^{n}\right)
    - q_+\left( \rho_{j-1/2}^{n}\right) - q_-\left( \rho_{j+1/2}^{n}\right) \biggr).
\end{equation}

\begin{lemma}
    \label{appendix_lmm1}
    For all $n \in \N$ and $j \in \Z$, we have
    \begin{equation}
        \label{appendix_bound1}
        \rho_{j-1/2}^{n} - \rho_{j+1/2}^{n} \leq \frac{1}{\lambda \mu} \quad \text{and} \quad D_j^n \leq \frac{1}{\lambda \mu}.
    \end{equation}
\end{lemma}

\begin{proof}
    Indeed, using first the uniform convexity of $f$ and then the CFL condition \eqref{1Bus_cfl}, we can write: 
    \[
        \left(\rho_{j-1/2}^{n} - \rho_{j+1/2}^{n} \right) \mu 
        \leq - \int_{\rho_{j+1/2}^{n}}^{\rho_{j-1/2}^{n}} f''(u) \d{u}
        \leq 2 \|f'\|_{\L{\infty}}
        \leq \frac{\Delta x}{\Delta t},
    \]

    from which we deduce \eqref{appendix_bound1}.
\end{proof}

\begin{lemma}
    \label{appendix_lmm2}
    Let $n \in \N$, $j \in \widehat{\Z}_{n+1}$, $\ds{a = \frac{\lambda \mu}{4}}$ and $\ds{\psi(x) = x - ax^2}$. Then
    \begin{equation}
        \label{appendix_bound2}
        D_j^{n+1} \leq \psi \left( \max\left\{D_{j-1}^n, D_j^n, D_{j+1}^n \right\} \right).
    \end{equation}
\end{lemma}

\begin{proof}
    We divide the proof in three steps.

    \textbf{Step 1: The function $\psi$} is nonnegative on $[0, 1/a]$ and nondecreasing on $[0, 1/(2a)]$. Note that by \eqref{appendix_bound1}, 
    $\ds{\max\left\{D_{j-1}^n, D_j^n, D_{j+1}^n \right\} \leq 1/(4a)}$, which will allow us to use the monotonicity of $\psi$.

    \textbf{Step 2.} We assume that 
    \begin{equation}
        \label{appendix_lmm2_eq1}
        \rho_{j+1/2}^{n} - \rho_{j+3/2}^{n} \geq 0 \quad \text{and} \quad \rho_{j-3/2}^{n} - \rho_{j-1/2}^{n} \geq 0,
    \end{equation}
    
    and we are going to prove that \eqref{appendix_bound2} holds. Using the uniform concavity assumption of $f$, we can write that 
    \begin{equation}
        \label{appendix_lmm2_eq2}
        \forall a,b \in [0,1], \quad 
        q_+(b) - q_+(a) \leq (b \wedge \overline{\rho} - a \wedge \overline{\rho}) f'(a \wedge \overline{\rho}) 
        - \frac{\mu}{2} (b \wedge \overline{\rho} - a \wedge \overline{\rho})^2.
    \end{equation}

    A similar inequality holds for $q_-$ as well. Using \eqref{mf_eo}, we obtain:
    \begin{equation}
        \label{appendix_lmm2_eq3}
        \begin{aligned}
            \rho_{j-1/2}^{n+1} - \rho_{j+1/2}^{n+1}
            & = \rho_{j-1/2}^{n} - \rho_{j+1/2}^{n} \\[5pt]
            & - \lambda \left( q_+\left( \rho_{j-1/2}^{n}\right) - q_+\left( \rho_{j-3/2}^{n}\right) - q_+\left( \rho_{j+1/2}^{n}\right) 
            + q_+\left( \rho_{j-1/2}^{n}\right) \right)  \\[5pt]
            & - \lambda \left( q_-\left( \rho_{j+1/2}^{n}\right) - q_-\left( \rho_{j-1/2}^{n}\right) - q_-\left( \rho_{j+3/2}^{n}\right) 
            + q_-\left( \rho_{j+1/2}^{n}\right) \right) \\[5pt]
            & = \rho_{j-1/2}^{n} - \rho_{j+1/2}^{n} \\[5pt]
            & + \lambda \left\{ \left( q_+\left( \rho_{j+1/2}^{n}\right) - q_+\left( \rho_{j-1/2}^{n}\right) \right) 
                + \left( q_+\left( \rho_{j-3/2}^{n}\right) - q_+\left( \rho_{j-1/2}^{n}\right) \right) \right. \\[5pt]
            & \left. 
            + \left (q_-\left( \rho_{j+3/2}^{n}\right) - q_-\left( \rho_{j+1/2}^{n}\right) \right) 
            + \left(q_-\left( \rho_{j-1/2}^{n}\right) - q_-\left( \rho_{j+1/2}^{n}\right) \right) \right\} \\[5pt]
            & \leq \rho_{j-1/2}^{n} - \rho_{j+1/2}^{n} \\[5pt] 
            & + \lambda \left( \rho_{j+1/2}^{n} \wedge \overline{\rho} - \rho_{j-1/2}^{n} \wedge \overline{\rho} \right) 
            f'(\rho_{j-1/2}^{n} \wedge \overline{\rho}) 
            - \frac{\lambda \mu}{2} \left( \rho_{j+1/2}^{n} \wedge \overline{\rho} - \rho_{j-1/2}^{n} \wedge \overline{\rho} \right)^2 \\[5pt]
            & + \lambda \left( \rho_{j-3/2}^{n} \wedge \overline{\rho} - \rho_{j-1/2}^{n} \wedge \overline{\rho} \right) 
            f'(\rho_{j-1/2}^{n} \wedge \overline{\rho}) 
            -\frac{\lambda \mu}{2} \left( \rho_{j-3/2}^{n} \wedge \overline{\rho} - \rho_{j-1/2}^{n} \wedge \overline{\rho} \right)^2 \\[5pt]
            & + \lambda \left( \rho_{j+3/2}^{n} \vee \overline{\rho} - \rho_{j+1/2}^{n} \vee \overline{\rho} \right) f'(\rho_{j+1/2}^{n} \vee \overline{\rho}) 
            -\frac{\lambda \mu}{2} \left( \rho_{j+3/2}^{n} \vee \overline{\rho} - \rho_{j+1/2}^{n} \vee \overline{\rho} \right)^2 \\[5pt]
            & + \lambda \left( \rho_{j-1/2}^{n} \vee \overline{\rho} - \rho_{j+1/2}^{n} \vee \overline{\rho} \right) f'(\rho_{j+1/2}^{n} \vee \overline{\rho}) 
            -\frac{\lambda \mu}{2} \left( \rho_{j-1/2}^{n} \vee \overline{\rho} - \rho_{j+1/2}^{n} \vee \overline{\rho} \right)^2,
        \end{aligned}
    \end{equation}

    where the last inequality comes from using \eqref{appendix_lmm2_eq2}. The proof now reduces to four cases, depending on the ordering of $\overline{\rho}$, 
    $\rho_{j-1/2}^{n}$ and $\rho_{j-1/2}^{n}$.

    \textbf{Case 1: $\overline{\rho} \geq \rho_{j-1/2}^{n}, \ \rho_{j+1/2}^{n}$.} Under assumption \eqref{appendix_lmm2_eq1}, we have 
    $\overline{\rho} \geq \rho_{j+3/2}^{n}$ as well. Inequality \eqref{appendix_lmm2_eq3} becomes:
    \begin{equation}
        \label{appendix_lmm2_eq4}
        \begin{aligned}
            \rho_{j-1/2}^{n+1} - \rho_{j+1/2}^{n+1} 
            & \leq \left(1 - \lambda f'(\rho_{j-1/2}^{n}) \right) \left( \rho_{j-1/2}^{n} - \rho_{j+1/2}^{n} \right) + 
            \lambda f'(\rho_{j-1/2}^{n}) \left( \rho_{j-3/2}^{n} \wedge \overline{\rho} - \rho_{j-1/2}^{n} \right) \\[5pt]
            & - \frac{\lambda \mu}{2} \left( \left( \rho_{j-1/2}^{n} - \rho_{j+1/2}^{n} \right)^2 
            + \left( \rho_{j-3/2}^{n} \wedge \overline{\rho} - \rho_{j-1/2}^{n} \right)^2 \right) \\[5pt]
            & \leq \left(1 - \lambda f'(\rho_{j-1/2}^{n}) \right) \left( \rho_{j-1/2}^{n} - \rho_{j+1/2}^{n} \right) + 
            \lambda f'(\rho_{j-1/2}^{n}) \left( \rho_{j-3/2}^{n} \wedge \overline{\rho} - \rho_{j-1/2}^{n} \right) \\[5pt]
            & - \frac{\lambda \mu}{4} \left( \left( \rho_{j-1/2}^{n} - \rho_{j+1/2}^{n} \right)^2 
            + \left( \rho_{j-3/2}^{n} \wedge \overline{\rho} - \rho_{j-1/2}^{n} \right)^2 \right) \\[5pt]
            & \leq \left(1 - \lambda f'(\rho_{j-1/2}^{n}) \right) \left( \rho_{j-1/2}^{n} - \rho_{j+1/2}^{n} \right) + 
            \lambda f'(\rho_{j-1/2}^{n} \left( \rho_{j-3/2}^{n} \wedge \overline{\rho} - \rho_{j-1/2}^{n} \right) \\[5pt]
            & - \frac{\lambda \mu}{4} \max\left\{\rho_{j-1/2}^{n} - \rho_{j+1/2}^{n}, 
            \rho_{j-3/2}^{n} \wedge \overline{\rho} - \rho_{j-1/2}^{n} \right\}^2,
        \end{aligned}
    \end{equation}

    where the last inequality comes from the bound: $\ds{a^2+b^2 \geq \max\{a, b\}^2}$. The CFL condition \eqref{1Bus_cfl} ensures that the two first 
    terms of the right-hand side of the last inequality are a convex combination of $\ds{\left( \rho_{j-1/2}^{n} - \rho_{j+1/2}^{n} \right)}$ and 
    $\ds{\left( \rho_{j-3/2}^{n} \wedge \overline{\rho} - \rho_{j-1/2}^{n} \right)}$. Consequently, inequality \eqref{appendix_lmm2_eq4} then becomes 
    \[
        \rho_{j-1/2}^{n+1} - \rho_{j+1/2}^{n+1} \leq 
        \psi \left( \max \left\{ \rho_{j-1/2}^{n} - \rho_{j+1/2}^{n}, \rho_{j-3/2}^{n} \wedge \overline{\rho} - \rho_{j-1/2}^{n} \right\} \right).
    \]

    Since $\ds{\rho_{j-3/2}^{n} \wedge \overline{\rho} - \rho_{j-1/2}^{n} \leq \rho_{j-3/2}^{n} - \rho_{j-1/2}^{n}}$, the monotonicity of $\psi$ ensures that 
    \[
        \begin{aligned}
            \rho_{j-1/2}^{n+1} - \rho_{j+1/2}^{n+1} 
            & \leq \psi \left( \max \left\{ \rho_{j-1/2}^{n} - \rho_{j+1/2}^{n}, \rho_{j-3/2}^{n} - \rho_{j-1/2}^{n} \right\} \right) \\[5pt]
            & \leq \psi \left( \max \left\{ D_{j-1}^n, D_j^n \right\} \right) \\[5pt]
            & \leq \psi \left( \max \left\{ D_{j-1}^n, D_j^n, D_{j+1}^n \right\} \right).
        \end{aligned}
    \]

    Since the right-hand side of this inequality is nonnegative, we can replace its left-hand side by $D_j^{n+1}$, which concludes the proof in this case.

    \textbf{Case 2: $\overline{\rho} \leq \rho_{j-1/2}^{n}, \ \rho_{j+1/2}^{n}$.} The proof of in this case similar to the last one so we omit the details.

    \textbf{Case 3: $\rho_{j+1/2}^{n} \leq \overline{\rho} \leq \rho_{j-1/2}^{n}$.} Under Assumption \eqref{appendix_lmm2_eq1}, we have the following ordering:
    \[
        \rho_{j+3/2}^{n} \leq \rho_{j+1/2}^{n} \leq \overline{\rho} \leq \rho_{j-1/2}^{n} \leq \rho_{j-3/2}^{n}.
    \]

    Inequality \eqref{appendix_lmm2_eq3} becomes 
    \[
        \begin{aligned}
            \rho_{j-1/2}^{n+1} - \rho_{j+1/2}^{n+1} 
            & \leq \rho_{j-1/2}^{n} - \rho_{j+1/2}^{n} 
            - \frac{\lambda \mu}{2} \left( (\rho_{j-1/2}^{n} - \overline{\rho})^2 + (\overline{\rho} - \rho_{j+1/2}^{n})^2 \right) \\[5pt]
            & \leq \rho_{j-1/2}^{n} - \rho_{j+1/2}^{n} - \frac{\lambda \mu}{4} (\rho_{j-1/2}^{n} - \rho_{j+1/2}^{n})^2,
        \end{aligned}
    \]

    where we used the inequality $2(a^2 + b^2) \geq (a + b)^2$. From here, we can conclude as in Case 1.

    \textbf{Case 4: $\rho_{j-1/2}^{n} \leq \overline{\rho} \leq \rho_{j+1/2}^{n}$.} Using the decomposition 
    \[
        \rho_{j-1/2}^{n} - \rho_{j+1/2}^{n}
        = (\rho_{j-1/2}^{n} \wedge \overline{\rho} - \rho_{j+1/2}^{n} \wedge \overline{\rho}) 
        + (\rho_{j-1/2}^{n} \vee \overline{\rho} - \rho_{j+1/2}^{n} \vee \overline{\rho}),
    \]

    inequality \eqref{appendix_lmm2_eq3} becomes
    \begin{equation}
        \label{appendix_lmm2_eq5}
        \begin{aligned}
            \rho_{j-1/2}^{n+1} - \rho_{j+1/2}^{n+1}
            & \leq \left(1 - \lambda f'(\rho_{j-1/2}^n) \right) 
            \left( \rho_{j-1/2}^{n} \wedge \overline{\rho} - \rho_{j+1/2}^{n} \wedge \overline{\rho} \right)
            + \lambda f'(\rho_{j-1/2}^n) \left( \rho_{j-3/2}^{n} \wedge \overline{\rho} - \rho_{j-1/2}^{n} \wedge \overline{\rho} \right) \\[5pt]
            & + \left(1 + \lambda f'(\rho_{j+1/2}^n) \right) 
            \left( \rho_{j-1/2}^{n} \vee \overline{\rho} - \rho_{j+1/2}^{n} \vee \overline{\rho} \right)
            - \lambda f'(\rho_{j+1/2}^n) \left( \rho_{j+1/2}^{n} \vee \overline{\rho} - \rho_{j+3/2}^{n} \vee \overline{\rho} \right) \\[5pt]
            & - \frac{\lambda \mu}{2} \left\{ \left( \rho_{j-1/2}^{n} \wedge \overline{\rho} - \rho_{j+1/2}^{n} \wedge \overline{\rho} \right)^2
            + \left( \rho_{j-3/2}^{n} \wedge \overline{\rho} - \rho_{j-1/2}^{n} \wedge \overline{\rho} \right)^2 \right. \\[5pt]
            & \left. + \left( \rho_{j-1/2}^{n} \vee \overline{\rho} - \rho_{j+1/2}^{n} \vee \overline{\rho} \right)^2 
            + \left( \rho_{j+1/2}^{n} \vee \overline{\rho} - \rho_{j+3/2}^{n} \vee \overline{\rho} \right)^2 \right\} \\[5pt]
            & \leq \left(1 - \lambda f'(\rho_{j-1/2}^n) \right) 
            \left( \rho_{j-1/2}^{n} \wedge \overline{\rho} - \rho_{j+1/2}^{n} \wedge \overline{\rho} \right)
            + \lambda f'(\rho_{j-1/2}^n) \left( \rho_{j-3/2}^{n} \wedge \overline{\rho} - \rho_{j-1/2}^{n} \wedge \overline{\rho} \right) \\[5pt]
            & + \left(1 + \lambda f'(\rho_{j+1/2}^n) \right) 
            \left( \rho_{j-1/2}^{n} \vee \overline{\rho} - \rho_{j+1/2}^{n} \vee \overline{\rho} \right)
            - \lambda f'(\rho_{j+1/2}^n) \left( \rho_{j+1/2}^{n} \vee \overline{\rho} - \rho_{j+3/2}^{n} \vee \overline{\rho} \right) \\[5pt]
            & - \frac{\lambda \mu}{2} \left\{ \left( \rho_{j-1/2}^{n} \wedge \overline{\rho} - \rho_{j+1/2}^{n} \wedge \overline{\rho} \right)^2 
            + \left( \rho_{j-1/2}^{n} \vee \overline{\rho} - \rho_{j+1/2}^{n} \vee \overline{\rho} \right)^2 \right\}.
        \end{aligned}
    \end{equation}

    The CFL condition \eqref{1Bus_cfl} and the ordering $\rho_{j+1/2}^{n} \leq \overline{\rho} \leq \rho_{j-1/2}^{n}$ result in 
    \[
        \left(1 - \lambda f'(\rho_{j-1/2}^n) \right) \left( \rho_{j-1/2}^{n} \wedge \overline{\rho} - \rho_{j+1/2}^{n} \wedge \overline{\rho} \right) \leq 0 
        \quad \text{and} \quad 
        \left(1 + \lambda f'(\rho_{j+1/2}^n) \right) \left( \rho_{j-1/2}^{n} \vee \overline{\rho} - \rho_{j+1/2}^{n} \vee \overline{\rho} \right) \leq 0
    \]

    so we can replace \eqref{appendix_lmm2_eq5} by 
    \[
        \begin{aligned}
            \rho_{j-1/2}^{n+1} - \rho_{j+1/2}^{n+1}
            & \leq \lambda f'(\rho_{j-1/2}^n) \left( \rho_{j-3/2}^{n} \wedge \overline{\rho} - \rho_{j-1/2}^{n} \wedge \overline{\rho} \right) 
            - \lambda f'(\rho_{j+1/2}^n) \left( \rho_{j+1/2}^{n} \vee \overline{\rho} - \rho_{j+3/2}^{n} \vee \overline{\rho} \right) \\[5pt]
            & - \frac{\lambda \mu}{2} \left\{ \left( \rho_{j-1/2}^{n} \wedge \overline{\rho} - \rho_{j+1/2}^{n} \wedge \overline{\rho} \right)^2 
            + \left( \rho_{j-1/2}^{n} \vee \overline{\rho} - \rho_{j+1/2}^{n} \vee \overline{\rho} \right)^2 \right\} \\[5pt]
            & \leq \frac{1}{2} \left( \left( \rho_{j-3/2}^{n} \wedge \overline{\rho} - \rho_{j-1/2}^{n} \wedge \overline{\rho} \right) 
            + \left( \rho_{j+1/2}^{n} \vee \overline{\rho} - \rho_{j+3/2}^{n} \vee \overline{\rho} \right) \right) \\[5pt]
            & - \frac{\lambda \mu}{4} \left\{ \left( \rho_{j-1/2}^{n} \wedge \overline{\rho} - \rho_{j+1/2}^{n} \wedge \overline{\rho} \right)^2 
            + \left( \rho_{j-1/2}^{n} \vee \overline{\rho} - \rho_{j+1/2}^{n} \vee \overline{\rho} \right)^2 \right\} \\[5pt]
            & \leq \psi \left( \max\left\{\left( \rho_{j-3/2}^{n} \wedge \overline{\rho} - \rho_{j-1/2}^{n} \wedge \overline{\rho} \right), 
            \left( \rho_{j-1/2}^{n} \vee \overline{\rho} - \rho_{j+1/2}^{n} \vee \overline{\rho} \right) \right\}\right),
        \end{aligned}
    \]

    and we exploit the monotonicity of $\psi$ to conclude.

    \textbf{Step 3: We no longer assume \eqref{appendix_lmm2_eq1}, and we get back to the general case.} Let us introduce
    \[
        u_{j-3/2}^n = \rho_{j-3/2}^n \vee \rho_{j-1/2}^n, \; u_{j-1/2}^n = \rho_{j-1/2}^n, \; u_{j+1/2}^n = \rho_{j+1/2}^n, \; 
        u_{j+3/2}^n = \rho_{j+3/2}^n \wedge \rho_{j-1/2}^n,
    \]

    and 
    \[
        u_{j-1/2}^{n+1} = H(u_{j-3/2}^n, u_{j-1/2}^n, u_{j+1/2}^n); \quad u_{j+1/2}^{n+1} = H(u_{j-1/2}^n, u_{j+1/2}^n, u_{j+3/2}^n).
    \]

    Using the monotonicity of $H$, we get:
    \[
        \begin{aligned}
            \rho_{j-1/2}^{n+1} - \rho_{j+1/2}^{n+1} 
            & = H(\rho_{j-3/2}^n, \rho_{j-1/2}^n, \rho_{j+1/2}^n) - H(\rho_{j-1/2}^n, \rho_{j+1/2}^n, \rho_{j+3/2}^n) \\[5pt]
            & \leq H(u_{j-3/2}^n, u_{j-1/2}^n, u_{j+1/2}^n) - H(u_{j-1/2}^n, u_{j+1/2}^n, u_{j+3/2}^n) = u_{j-1/2}^{n+1} - u_{j+1/2}^{n+1}.
        \end{aligned}
    \]

    Since $u_{j+1/2}^{n} - u_{j+3/2}^{n} \geq 0$ and $u_{j-3/2}^{n} - u_{j-1/2}^{n} \geq 0$, Step 2 ensures that 
    \[
        {\overset{\sim}{D}}_j^{n+1} 
        \leq \psi \left( \max\left\{{\overset{\sim}{D}}_{j-1}^n, {\overset{\sim}{D}}_j^n, {\overset{\sim}{D}}_{j+1}^n \right\} \right), \quad 
        {\overset{\sim}{D}}_j^n = \max\left\{ u_{j-1/2}^n - u_{j+1/2}^n, 0 \right\}.
    \]

    Clearly, 
    \[
        {\overset{\sim}{D}}_{j-1}^n \leq D_{j-1}^n, \quad {\overset{\sim}{D}}_j^n = D_j^n, \quad {\overset{\sim}{D}}_{j+1}^n \leq D_{j+1}^n.
    \]

    Using the monotonicity of $\psi$, we get:
    \[
        \rho_{j-1/2}^{n+1} - \rho_{j+1/2}^{n+1} 
        \leq u_{j-1/2}^{n+1} - u_{j+1/2}^{n+1}
        \leq \psi \left( \max\left\{D_{j-1}^n, D_j^n, D_{j+1}^n \right\} \right),
    \]

    concluding the proof.
\end{proof}

\newpage

{\small
  \bibliography{multibus}
  \bibliographystyle{abbrv}
}

\end{document}